\documentclass[12pt]{amsart}

\usepackage{amsmath}
\usepackage{amsthm}
\usepackage{amssymb}
\usepackage{amscd}
\usepackage{amsfonts}
\usepackage{amsbsy}
\usepackage{graphicx,epsf,epstopdf,color,pifont,subfigure,flafter,bm,amscd}

\usepackage{hyperref}
\hypersetup{colorlinks=true,citecolor=blue,%
        linkcolor=blue,urlcolor=blue,bookmarksnumbered=true,
        bookmarksopen=true,bookmarksopenlevel=1,breaklinks=true}

\newcommand{\R}{\ensuremath{\mathbb{R}}}
\newcommand{\C}{\ensuremath{\mathbb{C}}}

\def\epsilon{\varepsilon}
\newtheorem{theorem}{Theorem}[section]
\newtheorem{proposition}[theorem]{Proposition}
\newtheorem{lemma}[theorem]{Lemma}
\newtheorem{corollary}[theorem]{Corollary}

\numberwithin{equation}{section}

\numberwithin{figure}{section}

\begin{document}

\title[the cyclicity of the period
annulus]
{On the cyclicity of the period
annulus of quasi-homogeneous polynomial vector fields}

\author[Gavrilov, He and Xiao  ]
{Lubomir Gavrilov, Hongjin He,  ~~ Dongmei Xiao$^{\dag}$  }

\address{Dept of Math., University of Toulouse, France}\email{lubomir.gavrilov@math.univ-toulouse.fr}

\address{College of Mathematics and Statistics, Key Laboratory of Nonlinear Analysis and its Applications (Ministry of Education),
 Chongqing University, Chongqing, 401331,  China}\email{hehongjin000@126.com}

\address{School of Mathematical Sciences, CMA-Shanghai, Shanghai
	Jiao Tong University, Shanghai, 200240, China}
\email{xiaodm@sjtu.edu.cn}

\begin{abstract}
In this article, we study the number of limit cycles, bifurcating from the period annulus of any quasi-homogeneous polynomial vector fields with a center, under a one-parameter polynomial perturbation.  We first recharacterize quasi-homogeneous polynomial vector fields and its global center, then we establish an upper bound formula for the number of the isolated zeros of the $k$th order Melnikov function in terms of $k$, $max \{s_1,s_2\}$ and the degree $n$ of the perturbation by applying adapted Francoise's algorithm in conjunction with combinatorial techniques, where $(s_1,s_2)$ is the weight exponent of the quasi-homogeneous polynomial vector field. This extends relevant results presented in the literature [JDDE,21(2009)133-152] and [JDE, 276(2021)1-24].   As an application, we completely solve the limit cycle bifurcation problem of a perturbated quasi-homogeneous polynomial vector field.
\end{abstract}

\maketitle
\tableofcontents

\section{Introduction}
In this article, we investigate the general bifurcation theory of near quasi-homogeneous polynomial vector fields $X_\lambda$. More precisely, let $X_\lambda$, $\lambda \in (\R^N,0)$, be an analytic family of polynomial vector fields on $\R^2$, where $X_0$ is a quasi-homogeneous  vector field with a center and $(\R^N,0)$ is a small neighborhood of the origin in $\R^N$. An open period annulus  of $X_0$ is an open connected subset of the phase space $\R^2$, foliated by periodic orbits of $X_0$, homeomorphic  to
 the standard annulus $\mathbb{S}^1\times (0,1)$. The open annulus $U$ of the quasi-homogeneous vector field $X_0$ is therefore the punctured plane $\R^2\setminus \{(0,0)\}$.
 {\it The infinitesimal Hilbert 16th problem} on the period annulus $U$ is to find the uniform lower upper bound on the number of limit cycles of $X_\lambda$ for $\lambda$, which tend to $U$ when $\lambda\rightarrow 0$. This bound is called the {\it cyclicity} of $U$ with respect to the unfolding $X_\lambda$. More precisely, for an arbitrary compact set $K\subseteq U$, define the cyclicity $Cycl(K,X_\lambda)$ of $K$ as the maximal number of limit cycles of the vector field $X_\lambda$, which tend to the compact set $K$ as $\lambda\rightarrow 0$. The cyclicity $Cycl(U,X_\lambda)$ of
 an arbitrary invariant set (like the open period annulus $U$ of $X_0$)  is defined then by
$$
Cycl(U,X_\lambda)=\sup_{K\subseteq U}\{Cycl(K,X_\lambda): K\text{ is compact}\}.
$$

Determining the cyclicity of a period annulus for polynomial vector fields remains, in full generality, a largely open problem,
 closely related to the second part of the famous Hilbert 16th problem, see \cite{Rou1}. It is known that the cyclicity of an elementary homoclinic orbit (or an elementary polycycles in generic families) is finite, see \cite{gavr13, IY, Kal, Rou1, Rou2}. And the finiteness of the cyclicity holds true for open period annuli of the Hamiltonian and the generic Darbouxian vector fields, see \cite{gano08}. Iliev in \cite{Iliev} considered the cyclicity of the harmonic oscillator $H=\frac{1}{2}(x^2+y^2)$,  and Buica {\it et al.} in \cite{BGL} gave the estimation of the cyclicity  of Hamiltonian $H=(x^2+y^2)^m$. The cyclicity  of pertubated quasi-homogeneous Hamiltonian polynomial vector fields was established in \cite{FHX}. However, non-Hamiltonian quasi-homogeneous polynomial vector fields with an open period annuli are  usually not the generic Darbouxian vector fields.

 It follows from \cite[Theorem 1]{gavr08} and \cite[Proposition 2]{gano08} that for every open period annulus $U$ and unfolding $X_\lambda$ there always exists
an analytic arc $\epsilon\rightarrow \lambda(\epsilon)$, $\lambda(0)=0$,  such that
\begin{align}
\label{reduction}
Cycl(U,X_{\lambda(\epsilon)})=Cycl(U,X_\lambda), \varepsilon \in (\R,0) .
\end{align}
This observation transforms the cyclicity problem of a quasi-homogeneous polynomial vector field unfolding $X_\lambda$ with degree $n$ and  a finite number of parameters $\lambda$   to that  of a one-parameter family of vector fields $X_{\lambda(\epsilon)}$
\begin{equation}\label{eq2}
X_{\lambda(\epsilon)} :
\left\{
\begin{aligned}
&\dot{x}=P(x,y)+\sum_{i=1}^\infty\epsilon^iP_i(x,y),\\
&\dot{y}=Q(x,y)+\sum_{i=1}^\infty\epsilon^iQ_i(x,y),
\end{aligned}
\right.
\end{equation}
where
\begin{equation}\label{eq1}
X_0 :
\left\{
\begin{aligned}
&\dot{x}=P(x,y),\\
&\dot{y}=Q(x,y)
\end{aligned}
\right.
\end{equation}
is any quasi-homogeneous polynomial vector field with an open annuli $U$,  and $P_i,Q_i$ are polynomials of $(x,y)$ with degree not more than $n$.
Equivalently, it will be often convenient to consider the related foliation $\mathcal{F}(\omega_\epsilon)$ associated to the one-form
\begin{equation}\label{eq22}
\omega_\epsilon:=\sum_{i=0}^\infty \epsilon^i\omega_i, \ \ 0<|\epsilon|\ll 1
\end{equation}
where
$$
\omega_0=Q(x,y)dx-P(x,y)dy,\quad \omega_i=Q_i(x,y)dx-P_i(x,y)dy,
$$
and $\omega_0$ is a quasi-homogeneous polynomial one-form.

By \cite{LLNZ}, we know that system $ Q(x,y)dx-P(x,y)dy = 0$ has an analytic first integral $H$ on $U$, i.e. there exists an analytic integrating factor $\mu(x,y)$ on $U$ such that $dH=\mu\omega_0$. Generally, $H$ is  quasi-homogeneous analytic function but not a polynomial.   Thus, system \eqref{eq2} can be transformed to a perturbated Hamiltonian non-polynomial vector field.  The foliation $\mathcal{F}(\omega_\epsilon)$ on every compact set $K\subseteq U$ has the same leaves as the foliation associated to the perturbated Hamiltonian equation
\begin{equation}\label{eq221}
\mu\omega_\epsilon=dH+\sum_{i=1}^\infty \epsilon^i\mu\omega_i=0.
\end{equation}

To study the bifurcation of limit cycles from the period annulus $U$
for system \eqref{eq221} as $0<|\epsilon|\ll 1$, a fundamental approach is to analyze the \textit{Poincar\'e return map} defined on a transverse section $L$ to the period annulus (cf. \cite{F,FP,G, FGX}).
Parameterizing the transverse section by the restriction $h=H(x,y)|_L$ on it, we get
$$
U=\bigcup_{h_1<h<h_2} \gamma_h,\ \ h_1, h_2\in \mathbb{R},
$$
where  $\gamma_h$ is an oval of the level set $\{(x,y)|\ H(x,y)=h, h_1<h<h_2 \}$. The Poincar\'e map on $L$ can be expressed as
\begin{equation}
P(h,\epsilon)=h+\epsilon^kM_k(h)+\epsilon^{k+1}M_{k+1}(h)+\dots,
\end{equation}
where $M_k(h)\not\equiv 0$ as $h_1<h<h_2$. Here $M_k(h)$ is called {\it $k$th order Melnikov function} (cf. \cite{FGX}).
A limit cycle corresponds to an isolated fixed point of the Poincar\'e map $P$.
If a limit cycle tends to an oval $\gamma_h$ as $\varepsilon \to 0$, then $M_k(h)=0$. Therefore the number of zeros of $M_k$ (counted with multiplicity) provides an upper bound to the cyclicity $Cycl(U,X_\lambda(\varepsilon))$, which is the content of well known Poincar\'e-Pontryagin-Melnikov theorem. Hence, the key point to solve the infinitesimal Hilbert 16th problem for near quasi-homogeneous polynomial vector fields $X_\lambda$ on the period annulus $U$ is to study the number of zeros of Melnikov functions $M_k$ for all $k$. We note here that from the identity (\ref{reduction}) and the general form of the Melnikov functions $M_k$, see Theorem \ref{mainthm}, it follows that $Cycl(U,X_{\lambda(\varepsilon)})$ is finite and hence
$Cycl(U,X_\lambda)$ is finite too.  A significant contribution on the bifurcation of limit cycles from $U$ of  quasi-homogeneous polynomial vector fields was made by Li et al. in \cite{Li}, who provided a refined upper bound for the number of zeros of the first-order Melnikov function $M_1(h)$.

The higher order  Melnikov functions were intensively studied in the last three decades in many papers. We mention a non-exhaustive list of references which inspired the present paper \cite{F,FP,FGX,G,Iliev2, LianLiu, JMP, JZ, Rou1, ZhangLi} and references therein. The fundamental method to calculate the Melnikov functions of perturbated Hamiltonian system \eqref{eq221} is the {\it Francoise's algorithm} (cf. \cite{F}).
Let $\Omega$ be a one-form such that $\int_{\gamma_h}\Omega=0$. Then it is possible to find $a, R$ such that
\begin{equation}\label{cohomo}
\Omega=-adH+dR .
\end{equation}
It follows that on every level set $\{ (x,y): H(x,y)= h\}$ it holds
$$
R = \int \Omega,\ \  a= - \int \frac{d \Omega}{dH},
$$
where $ \frac{d\Omega}{dH}$ is the Gelfand-Leray form of $d\Omega$.
Finding explicit expressions for $a, R$ in $\R^2$  is, however,  a challenging  problem. When $X_0$ is quasi-homogeneous Hamiltonian polynomial vector fields, authors in \cite{FHX} observed the Gelfand-Leray form
$
\frac{d\Omega}{dH}=\frac{\Delta\Omega(s_2ydx-s_1xdy)}{mH},
$
where the Hamiltonian $H(x,y)$ is $(s_1,s_2)$-quasi-homogeneous functions of weight degree $m$ (see Section 2) and $d\Omega=\Delta\Omega dx\wedge dy$.  They then
 obtained the upper bounded with explicit expression on the number of zeros of the first non-zero Melnikov function $M_k$ of polynomial unfolding $X_\lambda$ with degree $n$ by applying the Francoise's algorithm.

 One of the primary motivations for this article is to study the bifurcation of limit cycles from the period annulus $U$ for  quasi-homogeneous polynomial vector fields unfolding $X_{\lambda}$. The unfolding $X_{\lambda}$ not only extends the first-order perturbation  in \cite{Li} but also  generalizes the quasi-homogeneous polynomial Hamiltonian  vector field and its  perturbations  in  \cite{FHX}.
       We adopt a  different approach from that in \cite{FHX} to realize the Francoise's algorithm of any order Melnikov functions. Our main objective is to address the calculation of the first non-zero Melnikov functions $M_k$ and to estimate the number of their zeros by
       completing a series of combinatorial counting tasks.

 The paper is organized as follows.

In section \ref{Quasi-homogeneous}, we first discuss the general facts on homogeneous and quasi-homogeneous planar differential systems, which are by no means part of mathematical folklore and  no appropriate reference appears to be available, to our knowledge. Then we re-characterize quasi-homogeneous polynomial vector fields with a center.
 In section \ref{highOrderMelnikov}, we recall the Francoise's algorithm,  and we further adapted the algorithm in generalized polar coordinates such that it can be  applied to a general quasi-homogeneous differential system, see Theorem \ref{thm-Fran3}.
 In Sections 2 and 3 of these preliminaries will be our main technical tool in the rest of the paper.
 The main result of this paper is finally formulated in Theorem \ref{mainthm},  in section \ref{AHproblem}. We establish an upper bound for the number of the zeros of the $k$th order Melnikov function in terms of $k$, $max \{s_1,s_2\}$ and the degree $n$ of the perturbation, which is equal to the cardinal of a set denoted by $\mathcal{S}_{n+1,k}^{(s_1,s_2,m)}$ minus one. This leads that the cyclicity $Cycl(U,X_\lambda)$ is finite. 
In the final section, we provide an example to explain how our algorithm works on calculation of Melnikov functions in a perturbated quasi-homogeneous polynomial differential system.

\section{ Quasi-homogeneous and homogeneous polynomial planar vector fields}
\label{Quasi-homogeneous}
In the section we present some general results on homogeneous and quasi-homogeneous planar differential systems, and re-characterize quasi-homogeneous polynomial vector fields with a center.
The classical solution of the center-focus problem in the homogeneous case goes back at least to  Forster \cite{fors37} in 1937, see Theorem \ref{ana}. A homogeneous polynomial planar differential system has a first integral of Darboux type  $\prod_k l_i^{\alpha_i}exp (Q)$, where $l_i$ are linear forms, $\alpha_i$ are real numbers, and $Q$ is a rational function. Based on this, we propose a more geometric solution of the center-focus problem in Theorem \ref{geom}. As every quasi-homogeneous polynomial system is a pull back of a homogeneous one, Theorem \ref{quasi}, then these results are easily re-formulated in the quasi-homogeneous case, see Theorem \ref{geom*} and Theorem \ref{thm-1}. We give finally a direct proof of Theorem \ref{thm-1} and two useful corrby making use of the Liapounov generalized trigonometric functions $Sn$ and $Cn$.

A planar polynomial differential system $\dot{x}=P(x,y), \dot{y}=Q(x,y)$
defines a singular planar foliation
\begin{equation}
\label{fol}
Q(x,y) dx - P(x,y) dy = 0.
\end{equation}
 A diffeomorphism of $\R^2$ is said to be  a {\it symmetry} of (\ref{fol}) provided that the foliation (\ref{fol}) is invariant under the action of the diffeomorphism. For fixed weights $s_1,s_2>0$, we consider
the one-parameter group of linear maps of $\R^2$, $\lambda \mapsto \Phi_\lambda$, $\lambda\in \mathbb{R}_+^*$, here $\mathbb{R}_+^*=(0,+\infty)$
\begin{equation}
\label{ql}
 \Phi_\lambda(x,y) = (\lambda^{s_1}x, \lambda^{s_2}y),\quad \lambda >0 .
\end{equation}
The foliation \eqref{fol} is said to be \emph{$(s_1,s_2)$-quasi-homogeneous}, if it is invariant under the action of the one-parameter group of diffeomorphisms (\ref{ql}).  When $s_1=s_2=1$, the quasi-homogeneous foliation \eqref{fol} is called \emph{homogeneous}.
 $P$ and $Q$ are called {\it $(s_1,s_2)$-quasi-homogeneous functions of weight degree $d_1$ and $d_2$}, respectively, if
$$
P (\lambda^{s_1}x, \lambda^{s_2}y) = \lambda^{d_1} P(x,y),\quad Q (\lambda^{s_1}x, \lambda^{s_2}y) = \lambda^{d_2} Q(x,y)
$$
for suitable real $d_1,d_2,m$ which satisfy $m=d_2+s_1=d_1+s_2$. It is not difficult to see that
\begin{equation}\label{plo}
\Phi_\lambda^*(\omega_0)=\lambda^m\omega_0,
\end{equation}
and naturally the one-form $\omega_0$ is called {\it a $(s_1,s_2)$-quasi-homogeneous one-form of weight degree $m$} if equality \eqref{plo} holds, in other words, system \eqref{eq1} is a $(s_1,s_2)$-quasi-homogeneous polynomial system of weight degree $m+1-s_1-s_2$ (cf. \cite{Li}).

The weights $s_1,s_2$ satisfy moreover a linear relation with integer coefficients and therefore without loss of generality\emph{ we shall assume from now on that $s_1,s_2$ are mutually prime positive integers and  $s_1\geq s_2$}.

As the foliation \eqref{fol} allows a one-parameter group of symmetries, then it can be integrated, see \cite{arn80}. The result is that\emph{ every quasi-homogeneous planar vector field has a first integral of Darboux type}. To make this statement more precise, we consider in  detail the simplest non-trivial case of the quadratic homogeneous vector field.
\subsection{Homogeneous vector fields}
Consider first a real quadratic homogeneous vector field
\begin{equation}
\label{quadratic}
\left\{
\begin{aligned}
\dot{x}&= a_0 x^2 + a_1 xy + a_2 y^2 \\
\dot{y}&= b_0 x^2 + b_1 xy + b_2 y^2.
\end{aligned}
\right.
\end{equation}
where the polynomials
$$P(x,y)= a_0 x^2 + a_1 xy + a_2 y^2, Q(x,y)= b_0 x^2 + b_1 xy + b_2 y^2$$ are mutually prime in the ring $\R[x,y]$. Denote
$$
-x Q(x,y)+ y P(x,y)= \Pi_{k=1}^3 (y_k x- x_k y)
$$
where $(x_k,y_k)$ are suitable non-zero complex vectors.
We have then the following
\begin{theorem}
\label{prop1}
The quadratic vector field (\ref{quadratic}) has a first integral as follows
 \begin{description}
\item[(a)] If the vectors $(x_1,y_1), (x_2,y_2)$ and $(x_3,y_3)$ are two by two independent, then
 (\ref{quadratic})  has a first integral of Darboux type
\begin{equation*}
\label{generic}
H(x,y) = \Pi_{k=1}^3 (y_k x- x_k y)^{\alpha_k}
\end{equation*}
where
$\alpha_1 \neq 0, \alpha_2 \neq 0, \alpha_3 \neq 0, \alpha_1+\alpha_2 + \alpha_3 \neq 0 .$
\item[(b)] If two of the vectors $(x_k,y_k)$ are co-linear, but the third one is linearly independent
 then, up to a linear change of the variables,  (\ref{quadratic})  has a first integral of Darboux type
$$
H(x,y) = (y_1 x- x_1 y)^{\alpha_1} (y_2 x- x_2 y)^{\alpha_2} e^\frac{p_2(x,y)}{(y_2 x- x_2 y)}
$$
 where
$ \alpha_1 \neq 0,   \alpha_1 + \alpha_2 \neq 0
$
and $p_2$ is a homogeneous polynomial of degree one, not divisible by $y_2 x- x_2 y$.
\item[(c)] If $(x_1,y_1), (x_2,y_2) , (x_3,y_3) $ are aligned, then, up to a linear change of the variables,  (\ref{quadratic})  has a first integral of Darboux type
$$
H(x,y) =  (y_1 x- x_1 y)^{\alpha_1}e^\frac{p_1(x,y)}{(y_1 x- x_1 y)^2}
$$
where $p_1$ is a homogeneous polynomial of degree two, not divisible by $y_1 x- x_1 y$.
\end{description}

\end{theorem}
\proof
Introduce the new variable   $ w = \frac xy$. We can write  (\ref{fol}) in the form
$$
\frac{dx}{dy} = w + y  \frac{d w}{dy} =
 \frac{P(x,y) }{Q(x,y)}, \;
$$
or equivalently
\begin{align*}
\frac{dy}{y} = & (\frac{P(x,y)}{Q(x,y) } - \frac xy) ^{-1} dw \\
= & \frac{y Q(x,y)}{ -x Q(x,y) + y P(x,y)} dw \\
= &  \frac{ Q(w,1)}{ -w Q(w,1) +  P(w,1)} dw .
\end{align*}
In the case $({\bf a})$ the polynomial $  -w Q(w,1) +  P(w,1)$ has three distinct roots given by $\frac{x_k}{y_k}$, $k=1,2,3$ (assuming that $y_k\neq 0$).
Therefore
$$
\frac{ Q(w,1)}{ -w Q(w,1) +  P(w,1)} = \frac{  \alpha_1}{w- \frac{x_1}{y_1}} + \frac{\alpha_2}{w- \frac{x_2}{y_2}} +  \frac{\alpha_3}{w- \frac{x_3}{y_3}}.
$$
A simple residue calculus implies $ \alpha_1+\alpha_2+\alpha_3 = -1$ and
$$
\alpha_i=\frac{Q(\frac{x_i}{y_i},1)}{\Pi_{j\neq i}(\frac{x_i}{y_i}-\frac{x_j}{y_j})}\neq 0.
$$
Hence, the above foliation has a first integral
\begin{align*}
\Pi_{k=1}^3 (w- \frac{x_k}{y_k})^{\alpha_k} y^{\alpha_1+\alpha_2+\alpha_3} = \Pi_{k=1}^3 (x- \frac{x_k}{y_k} y)^{\alpha_k},
\end{align*}
which implies the result in the "generic" case ({\bf a}). The remaining cases are studied in a similar way.
 \endproof
 More generally, let
 \begin{equation}
 \label{ic}
 P\frac{\partial}{\partial x} + Q\frac{\partial}{\partial y}
 \end{equation}
 be a homogeneous vector field of arbitrary degree $d\geq 2$.
 By degree of (\ref{ic}) here we mean the number of tangency points which has the field with a generic straight line in $\C^2$. This geometric definition implies that if  (\ref{ic}) is of degree $d$, then $-x Q(x,y) + y P(x,y)$ is of degree \emph{exactly} $d+1$.

The proof of Theorem \ref{prop1} requires some remarks, which are valid for any homogeneous vector field of degree $d \geq 2$ and which we list below.
Namely :
 \begin{itemize}
\item
The zero locus of the homogeneous polynomial of degree $d+1$
$$-x Q(x,y) + y P(x,y) = \det \left(\begin{array}{cc}
P & Q \\
x & y\end{array}\right) = \Pi_{k=1}^{d+1} (y_k x- x_k y)
$$ is a reducible over $\C$ curve
$$
\Gamma= \{ (x,y)\in \C^2 : -x Q(x,y) + y P(x,y)  = 0 \}
$$
along which the vector field
$P\frac{\partial}{\partial x} + Q\frac{\partial}{\partial y}$ and the radial vector field $x\frac{\partial}{\partial x} + y\frac{\partial}{\partial y}$ are mutually tangent.
The radial vector field is, however, tangent to $\Gamma$. Therefore, $\Gamma$ is tangent to the homogenous vector field $P\frac{\partial}{\partial x} + Q\frac{\partial}{\partial y}$,  (\ref{ic}). Of course, this is not surprising, because the polynomial defining $\Gamma$ is an inverse integrating factor of (\ref{ic}).
\item
The real foliation (\ref{fol}) induces a complex foliation on $\C^2$, which is completed to a singular foliation on the projective plane $\mathbb P^2$ with homogeneous coordinates $[x:y:z]$. The infinite line $z=0$ is a leaf of the foliation, and the degree $d+1$ invariant line $\Gamma$
intersects the infinite line $z=0$ into $d+1$ singular points, counted with multiplicities, with projective coordinates
$[x_k:y_k:0]$. The case ({\bf a}) in Theorem \ref{prop1} corresponds to $d+1$ distinct singular points along $\{z=0\}$ and the remaining cases correspond to multiple singular points.
 \item
If the foliation (\ref{fol}) has multiple singular points along $\{z=0\}$, then it is a \emph{limit}, under of an appropriate deformation in the parameter space of foliations with $d+1$ distinct simple singular points along $\{z=0\}$.
\item
If the foliation (\ref{fol}) has $d+1$ distinct simple singular points $[x_k:y_k:0]$ along $\{z=0\}$, then as in the case $d=2$, it has a first integral
\begin{align}
\label{gen}
H(x,y) = \Pi_{k=1}^{d+1} (y_k x- x_k y)^{\alpha_k},
\end{align}
where $\alpha_k$ are appropriate non-zero complex numbers.
\item
In the remaining \emph{limiting} cases,  when some of the singular points at infinity are multiple, the invariant curve $\Gamma$ takes the form
\begin{align}
\label{nongen}
 \Gamma = \{(x,y)\in \C^2:  \Pi_{k} (y_k x- x_k y)^{n_k+1}=0 \},
 \end{align}
 where $\sum_k (n_k+1) = d+1 $. The first integral $H$ of the foliation (\ref{fol}) takes the form
$$
H(x,y) = \Pi_{k} (y_k x- x_k y)^{\alpha_k} e^{\frac{p_{k}(x,y)}{(y_k x- x_k y)^{n_k}}},
$$
 where $\alpha_k$ are non-zero complex numbers, $n_k+1$ is the multiplicity of the invariant line
 $\{y_k x- x_k y = 0 \}$, and  $p_{k}(x,y)$ is an arbitrary homogeneous polynomial of degree $n_k$ not divisible by $y_k x- x_k y$.
\end{itemize}
 The above considerations imply a solution of the classical center-focus problem for homogenous vector fields as follows. Let (\ref{ic}) be a homogeneous vector field of degree $d\geq 2$. Assuming that (\ref{ic}) is real and has a center, implies that it has no real invariant lines and
 that if $\{y_k x- x_k y = 0\}$ is a complex invariant line with $x_k/y_k \not \in \R$, then  $\{\overline{ y_k} x- \overline{ x_k} y = 0\}$ is an invariant line too. The curve $\Gamma$ is therefore of even degree and the first integral in full generality takes the form
\begin{align}
\label{firstintegral}
\Pi_{k=1}^{(d+1)/2} (y_k x- x_k y)^{\alpha_k}e^{\frac{p_{k}(x,y)}{(y_k x- x_k y)^{n_k}}}
(\overline{y_k} x- \overline{x_k} y)^{\overline{\alpha_k}} e^{\frac{\overline{p_{k}}(x,y)}{(\overline{y_k} x- \overline{x_k} y)^{n_k}}} .
\end{align}
The vector field has a center at the origin if and only if $H$ is a single valued analytic function on the punctured real plane $\R^2\setminus \{(0,0)\}$.
In the special case $\alpha_k\in \R$,  this is obviously true. Assume now that some $\alpha_k$ is not real. More specifically, consider the function
$$
\R^2\setminus \{(0,0)\} \to \R : (x,y) \mapsto (x+iy)^{\alpha i} (x-iy)^{-\alpha i}, \alpha \in \R .
$$
 If we denote $\varphi$ the argument of the complex variable $z=x+iy$, then
 $$
 (x+iy)^{\alpha i} (x-iy)^{-\alpha i} = e^{-2\alpha \varphi},
 $$
 which shows that when $(x,y)$ makes one turn around $(0,0)$ on $\R^2$ in a positive direction, then  $(x+iy)^{\alpha i} (x-iy)^{-\alpha i}$
 is multiplied by $e^{-4\alpha \pi}$. Similarly,  when $(x,y)$ makes one turn around $(0,0)$ on $\R^2$ in a positive direction, the argument of $y_k x- x_k y$
 \begin{itemize}
\item increases by $2\pi$, if the imaginary part of $y_k/x_k$ is positive
\item decreases by $2\pi$, if the imaginary part of $y_k/x_k$ is negative.
\end{itemize}
 Note that if the argument of $y_k x- x_k y$ increases, then the argument of $\overline{y_k} x- \overline{x_k} y$ decreases and vice versa.
     We obtain the following
 \begin{proposition}
 The  real function
 $$
 H(x,y) = \Pi_k (y_k x- x_k y)^{\alpha_k} (\overline{y_k} x- \overline{x_k} y)^{\overline{\alpha_k}},\  Im (y_k/x_k) >0
 $$
 is a single valued analytic function on $\R^2\setminus \{(0,0)\}$
if and only if $Im(\sum_k  \alpha_k ) = 0$.
 \end{proposition}
Taking finally into consideration that
\begin{align*}
e^{\frac{p_{k}(x,y)}{(y_k x- x_k y)^{n_k}}} e^{\frac{\overline{p_{k}}(x,y)}{(\overline{y_k} x- \overline{x_k} y)^{n_k}}}
&=
\exp[ \frac{p_{k}(x,y)}{(y_k x- x_k y)^{n_k}} +\frac{\overline{p_{k}}(x,y)}{(\overline{y_k} x- \overline{x_k} y)^{n_k}}] \\
 &= \exp[ \frac{q_k(x,y)} {\lvert y_k x- x_k y \rvert^{2 n_k}}]
\end{align*}
where $q_k(x,y)$ is a suitable homogeneous degree $2 n_k$ real polynomial, we conclude that
the above this expression is analytic on $\R^2\setminus \{(0,0)\}$. We obtain therefore
\begin{theorem}[geometric solution of the center-focus problem]
\label{geom}
The homogeneous polynomial vector field (\ref{ic}) has a center at the origin if and only if the following conditions are satisfied
\begin{enumerate}
\item [(i)] The inverse integrating factor $yP(x,y)-xQ(x,y)$ does not vanish on the punctured plane $\R^2 \setminus \{(0,0)\}$;
\item [(ii)] $Im(\sum_k  \alpha_k ) = 0 $ where $ \alpha_k$ are the exponents of the first integral (\ref{firstintegral}), with the convention that $Im(x_k/y_k)>0$ .
\end{enumerate}
\end{theorem}
 To apply the above Theorem we need to compute the first integral of (\ref{firstintegral}). A more intrinsic definition of the exponents $\alpha_k$ is the following. The one-form
$$
\frac{P(x,y) dy - Q(x,y) dx}{yP(x,y)-xQ(x,y)}
$$
is logarithmic, and $\alpha_k$ are the residues along its complex invariant lines.

The counterpart of this geometric characterization of center is the following classical result, which goes back at least to Forster \cite[1937]{fors37}, see also \cite[p.80]{nest60}. Namely, in polar coordinates
$$
x= r \cos(\theta), y= r \sin (\theta)
$$
the degree $d$ homogeneous foliation $P(x,y) dy - Q(x,y) dx = 0$ takes the form
$$
G(\theta) dr - r F(\theta) d \theta = 0
$$
with solutions
$$
r(\theta) = r_0 exp(\int_{\theta_0}^\theta \frac{F(\theta)}{G(\theta)} d\theta),
$$
where
\begin{align*}
F(\theta)&=\cos(\theta) P(\cos(\theta), \sin(\theta))+\sin(\theta) Q(\cos(\theta),\sin(\theta)),\\
G(\theta)&=\cos(\theta) Q(\cos(\theta),\sin(\theta))-\sin(\theta) P(\cos(\theta),\sin(\theta)).
\end{align*}
Therefore we have
\begin{theorem}[analytic solution of the center-focus problem]
\label{ana}
The degree $d$ homogeneous foliation $$P(x,y) dy - Q(x,y) dx = 0$$ has a center if and only if
\begin{itemize}
\item[(i)] The polynomial $yP(x,y)-xQ(x,y)$ does not vanish on the punctured plane $\R^2\setminus \{(0,0)\} $;
\item[(ii)]
$$
\int_{0}^{2\pi} \frac{F(\theta)}{G(\theta)}d\theta =0.
$$
\end{itemize}

\end{theorem}

 \subsection{Quasi-homogeneous vector fields}

It is possible to follow the above approach to integrate any  $(s_1,s_2)$-quasi-homogeneous planar differential system, after introducing a new variable
 $w=\frac{ x^{s_2} }{y^{s_1}}$. However, we do not need to do this. It is enough to observe that \emph{every quasi-homogeneous differential system is the pullback of a homogeneous one},  which can also be referred to Theorem 1 in \cite{TangZhang} and Theorem 3 in \cite{ZhangYu}. To be more precise, consider a planar foliation $P(x,y) dy - Q(x,y) dx = 0$ defined by two homogeneous polynomials $P, Q$. Upon substituting
\begin{equation}
\label{xys}
 x \to x^{s_2} , y \mapsto y^{s_1}
 \end{equation}
 we obtain a quasi-homogeneous planar foliation of weighted degrees $(s_1,s_2)$. We claim that the converse also holds :
  \begin{theorem}[pull back]
 \label{pb}
 Every $(s_1,s_2)$-quasi-homogeneous planar foliation
 $$
 P(x,y) dy - Q(x,y) dx = 0
 $$
 is a pull-back of a homogeneous planar foliation
 $$P^*(x,y) dy - Q^*(x,y) dx = 0$$
 ($P,Q$ are homogeneous polynomials of the same degree) under the map (\ref{xys}).
 \end{theorem}
 Let us prove first the following two elementary Propositions.

 \begin{proposition}
 \label{quasi}
 Let $H(x,y)$ be a quasi-homogeneous polynomial of weighted degree $(s_1,s_2)$, where $s_1,s_2$ are mutually prime positive integers. Then
\begin{equation}
\label{H}
  H(x,y) = x^{m_1}y^{m_2} \Pi_{k} (y_k x^{s_2}- x_k y^{s_1})^{n_k}
\end{equation}
  where $m_1\geq 0, m_2 \geq 0, n_k >0$ are integers, and $(x_k,y_k)$ are non zero complex vectors.
 \end{proposition}
 \proof
  If we impose the additional condition that the polynomial $H$ is real, and has no non-trivial irreducible components in $\R[x,y]$, and assuming the above Proposition,   we get that the vectors $(x_k,y_k)$ appear in complex-conjugate pairs, and that $m_1=m_2=0$ which is
 the claim of \cite[Proposition 6]{FHX}. The proof of the  Proposition \ref{quasi} is therefore similar to the proof of Proposition 6 \cite{FHX} (and is omitted).
 \endproof
 \begin{proposition}
 \label{quasir}
 Let $H(x,y)$ be a $(s_1,s_2)$-quasi-homogeneous rational function of weighted degree $0$.
 Then there exists a homogeneous rational function $H^*(x,y)$ of degree $0$, such that $$H(x,y)= H^*(x^{s_2},y^{s_1}).$$
 \end{proposition}
\proof
A rational $(s_1,s_2)$-quasi-homogeneous function is a ratio of two  $(s_1,s_2)$-quasi-homogeneous polynomials of the same weighted degree. Applying Proposition \ref{quasi} we obtain that $H$ is a product of a rational function in $x^{s_2},y^{s_1}$ and the term $x^{n_1}y^{n_2}$ where $n_1,n_2$ are integers (not necessarily positive).
As the weighted degree of $H$ is zero, then $s_1s_2$ divides $s_1 n_1 + s_2 n_2$ which implies that $s_2$ divides $n_1$ and that $s_1$ divides $n_2$. We conclude that
$H$ is in fact a rational function in $x^{s_2},y^{s_1}$ and $H^*$ defined by $H(x,y)= H^*(x^{s_2},y^{s_1})$ is a homogeneous rational function.
\endproof
\proof[Proof of Theorem \ref{pb}]
If the complex planar foliation
$$
P(x,y) dy - Q(x,y) dx = 0
$$
 is quasi-homogeneous, then the polynomials $yP(x,y)$ and $xQ(x,y)$ are quasi-homogeneous of the same weighted degree.
 Applying Proposition \ref{quasir} to $yP, xQ$ we get that the quasi-homogeneous planar foliation can be written in the form
 $$
 A(x^{s_2},y^{s_1}) \frac{dy}{y} - B(x^{s_2},y^{s_1}) \frac{dx}{x} = 0, A,B \in \R[x,y] .
 $$
 Finally we note that
 $$
 s_2\frac{dx}{x} = \frac{dx^{s_2}}{x^{s_2}}, s_1\frac{dy}{y} = \frac{dy^{s_1}}{y^{s_1}}.
 $$
 It means that the quasi-homogeneous foliation $P(x,y) dy - Q(x,y) dx = 0$ is the pull-back of the homogeneous planar foliation
 $$
 s_2xA(x,y)dy-s_1yB(x,y)dx=0
 $$
 under the map (\ref{xys}).
 \endproof

 Assume that $P(x,y) dy - Q(x,y) dx = 0$  is a $(s_1,s_2)$-quasi-homogeneous foliation. According to Theorem \ref{pb} it is a pull back of a homogeneous degree $d$ foliation $P^*(x,y) dy - Q^*(x,y) dx = 0$ under the map $x\to x^{s_2}, y\to y^{s_1}$.

 We are ready to formulate the analogues of Theorem \ref{geom} and Theorem \ref{ana}  in this case.
 \begin{theorem}[geometric solution of the center-focus problem]
 \label{geom*}

 The $(s_1,s_2)$ quasi-homogeneous foliation $P(x,y) dy - Q(x,y) dx = 0$ has a center if and only if one of the following conditions are satisfied
 \begin{itemize}
\item [(i)]   $s_1,s_2$ are odd and the homogeneous foliation $$P^*(x,y) dy - Q^*(x,y) dx = 0$$ has a center,
or
\item[(ii)] one of the weights $s_1, s_2$ is odd, the other is even, and  the homogeneous foliation $$P^*(x,y) dy - Q^*(x,y) dx = 0$$ has a center or a focus.
\end{itemize}
 \end{theorem}
 \proof
 The proof is straightforward. If $s_1,s_2$ are odd, the map $x\to x^{s_2}, y\to y^{s_1}$ is a diffeomorphism of $\R^2\setminus \{(0,0)\}$. If $s_1$ is even but $s_2$ is odd the image of the map $x\to x^{s_2}, y\to y^{s_1}$ is the half plane $y>0$. Therefore if the homogeneous foliation  $P^*(x,y) dy - Q^*(x,y) dx = 0$ has an invariant real line, then the quasi-homogeneous foliation has a separatrix and hence has no center. It remains to consider the case when the homogeneous foliation has a center or a focus. But in this case the quasi-homogeneous foliation can not have a real separatrix. Therefore it has a center or a focus. Note however that it is also reversible, it is invariant under the involution $(x,y)\to(x,-y)$, and hence has a center.\endproof




Using once again the Theorem \ref{pb} (pull back) we obtain the following analogue of Theorem \ref{ana}. We just use on the place of the usual polar coordinates by the generalized Liapounov coordinates \cite{L,Li}
$$
x= r Cs(\theta), y= r Sn(\theta)
$$
and substitute them in Theorem \ref{ana}.
\begin{theorem}[analytic solution of the center-focus problem]
\label{thm-1}

A $(s_1,s_2)$-quasi-homogeneous system of weight degree $m$  has a center  if and only if the following conditions hold.
\begin{itemize}
\item[(i)] $G(x,y)=s_1xQ(x,y)-s_2yP(x,y)$ does not vanish on the punctured plane
$\mathbb{R}^2\setminus \{(0,0)\}$;
\item[(ii)]
$$
\int_0^\tau\frac{F(\theta)}{G(\theta)}d\theta=0,
$$
where
\begin{equation}
\label{eq-fg}
\begin{aligned}
&F(\theta)=Cs^{2s_2-1}(\theta) P(Cs(\theta), Sn(\theta))+Sn^{2s_1-1}(\theta) Q(Cs(\theta),Sn(\theta)),\\
&G(\theta)=s_1 Cs(\theta) Q(Cs(\theta),Sn(\theta))-s_2 Sn(\theta) P(Cs(\theta),Sn(\theta)),
\end{aligned}
\end{equation}
and $Sn, Cs$ are the Lyapunov sine and cosine $(s_1,s_2)-$ trigonometric functions respectively with a period  $\tau$.
\end{itemize}
\end{theorem}
When $P$ and $Q$ are coprime, the proof of Theorem \ref{thm-1} was earlier proved in \cite{Li}. If $P$ and $Q$ are not coprime, the proof is straightforward again. To save space, we omit it.

\begin{corollary}\label{coro1}
If the origin $(0,0)$ is a center of a $(s_1,s_2)$ quasi-homogeneous system of degree $m$, then $s_1+s_2+m-1$ is even.
\end{corollary}

\begin{corollary}\label{coro2}
The  $(s_1,s_2)$  quasi-homogeneous polynomial system with a center has either one  pair of infinite equilibria on the Poincar\'e sphere,
or it has no infinite equilibria at all, see Figure \ref{pp}. Furthermore, the system has no infinite equilibria  if and only if it is homogeneous,  $s_1=s_2=1$.
\end{corollary}

 \begin{figure}[h]
 \centering
 \includegraphics[scale=0.35]{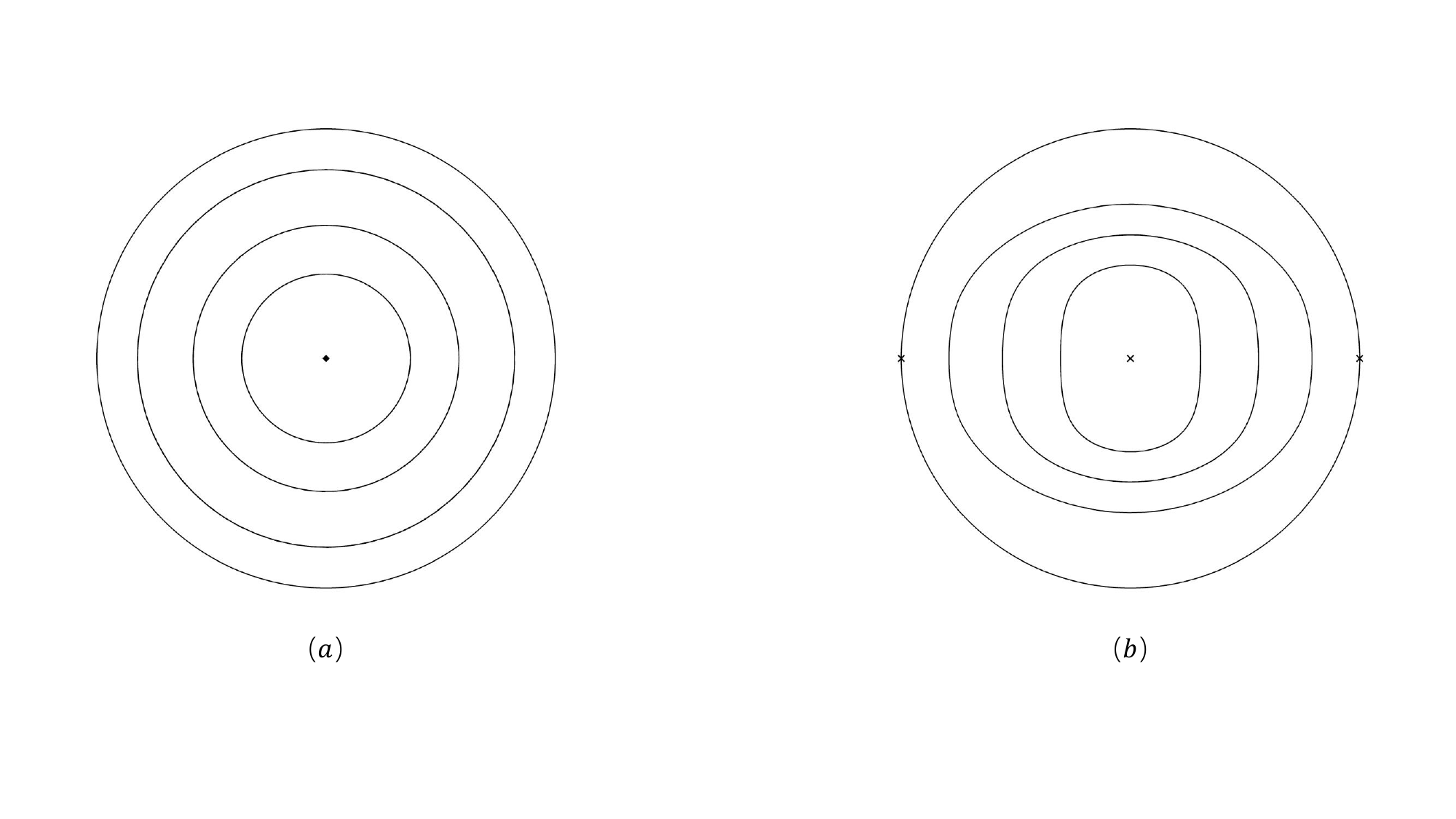}
 \caption{Global topological phase portraits of quasi-homogeneous polynomial systems with a center at the origin. }
 \label{pp}
 \end{figure}

For convenience of the reader, we first provide some basic information about the Lyapunov trigonometric functions, then we give  the proof of Corollary \ref{coro1} and \ref{coro2}.

Following \cite{L}, we recall the Lyapunov $(s_1,s_2)-$trigometric functions. Let $z(t)=Cs(t),w(t)=Sn(t)$ be the solution of equation
\begin{equation}\label{eq-qtri}
\dot{z}=-w^{2s_1-1},\quad \dot{w}=z^{2s_2-1}
\end{equation}
with initial $z(0)=s_1^{-\frac{1}{2s_2}},w(0)=0$, where $s_1,s_2\geq 1$ are the positive integers.
\begin{proposition}\label{pro-tri}
$Cs,Sn$ are the periodic functions with period $\tau$, and
\begin{equation*}
s_1Cs^{2s_2}(t)+s_2Sn^{2s_1}(t)=1.
\end{equation*}
where $\tau=2s_1^{-\frac{1}{2s_2}}s_2^{-\frac{1}{2s_1}}B(\frac{1}{2s_1},\frac{1}{2s_2})$ and $B$ is the $\beta$ function
\begin{equation*}
B(\alpha,\beta)=\int_0^1 (1-x)^{\alpha-1}x^{\beta-1} dx.
\end{equation*}
\end{proposition}
\begin{proof}
It is easily checked that the equation \eqref{eq-qtri} has the first integral $H(z,w)=s_1z^{2s_2}+s_2w^{2s_1}$, which implies
$$
H(Cs(t),Sn(t))=s_1Cs^{2s_2}(t)+s_1Sn^{2s_2}(t)=H(Cs0,Sn0)=H(s_1^{-\frac{1}{2s_2}},0)=1.
$$
Since $O=(0,0)$ is a local minimum of function $H(z,w)$, the origin $O$ is a center (in fact, a global center) of system \eqref{eq-qtri} and functions $Cs(t),Sn(t)$ are both periodic. Denote the period of function $Cs(t),Sn(t)$ by $\tau$. Let $x=s_1 z^{2s_2}$, we have
\begin{equation*}
\begin{aligned}
\tau&=-\oint_{H(z,w)=1}\frac{1}{w^{2s_1-1}}dz \\
&=4s_2^{-\frac{1}{2s_1}+1}\int_{0}^{s_1^{-\frac{1}{2s_2}}}(1-s_1z^{2s_2})^{\frac{1}{2s_1}-1}dz\\
&=4s_2^{-\frac{1}{2s_1}+1}\int_{0}^{1}(1-x)^{\frac{1}{2s_1}-1}d(s_1^{-\frac{1}{2s_2}}x^\frac{1}{2s_2})\\
&=2s_1^{-\frac{1}{2s_2}}s_2^{-\frac{1}{2s_1}}\int_0^1 (1-x)^{\frac{1}{2s_1}-1}x^{\frac{1}{2s_2}-1} dx\\
&=2s_1^{-\frac{1}{2s_2}}s_2^{-\frac{1}{2s_1}}B(\frac{1}{2s_1},\frac{1}{2s_2}).
\end{aligned}
\end{equation*}
\end{proof}
By proposition \ref{pro-tri}, functions $Cs,Sn$ can be well defined on the circle $\mathbb{R}/{(\tau\mathbb{Z})}\simeq\mathbb{S}^1$. Clearly, we use $\theta\in\mathbb{S}^1$ as the variable of $Cs$ and $Sn$. It is no hard to see $Cs(\theta)=\cos\theta,Sn(\theta)=\sin\theta$ when $s_1=s_2=1$. Furthermore, we can prove Lyapunov $(s_1,s_2)-$trigometric functions has the same properties to trigometric functions.
\begin{proposition}\label{tri-pro}
\begin{itemize}
\item[(i)] $Cs(\frac{\tau}{2}-\theta)=-Cs(\theta)$, $Sn(\frac{\tau}{2}-\theta)=Sn(\theta)$;\\
\item[(ii)] $Cs(\frac{\tau}{2}+\theta)=-Cs(\theta)$, $Sn(\frac{\tau}{2}+\theta)=-Sn(\theta)$;\\
\item[(iii)] $Cs(\tau-\theta)=Cs(\theta)$, $Sn(\tau-\theta)=-Sn(\theta)$.
\end{itemize}
\end{proposition}
\begin{proof}
These statements can be proved by the existence and uniqueness of the solutions of system \eqref{eq-qtri}. We only prove (i) for saving space. Note that
$$
(z,w)=(z_1(\theta),w_1(\theta))=(-Cs(\frac{\tau}{2}-\theta),Sn(\frac{\tau}{2}-\theta))
$$
and
$$
(z,w)=(z_2(\theta),w_2(\theta))=(Cs(\theta),Sn(\theta))
$$
are both the solutions of equation \eqref{eq-qtri}. We claim
$$
z_1(0)=z_2(0),w_1(0)=w_2(0).
$$
If it is true, then by the existence and uniqueness of the solutions of system \eqref{eq-qtri}, statement (i) is proved. In fact, system \eqref{eq-qtri} is invariant under transformation $(z,w)\mapsto(-z,-w)$ and has a first integral $H(z,w)=s_1 z^{2s_2}+s_2 w^{2s_1}$, thus the level set $H(x,y)=h$ is symmetric about the origin $O$. It implies that
$$
z_1(0)=-Cs(\frac{\tau}{2})=Cs(0)=z_2(0),\quad w_1(0)=Sn(\frac{\tau}{2})=Sn(0)=w_2(0).
$$
\end{proof}

\begin{lemma}\label{lem4}
Suppose $F(x,y)$ is a homogeneous polynomial. If $F(x,y)$ is also a $(s_1,s_2)$-quasi-homogeneous polynomial, where $(s_1,s_2)\neq (1,1)$, then $F$ is a monomial.
\end{lemma}
\begin{proof}
Suppose homogeneous polynomial $F$ has two monomials $a_1x^{n_{a_1}}y^{m_{a_1}}$ and $a_2x^{n_{a_2}}y^{m_{a_2}}$ with the same weight degree, that is,
\begin{equation*}
\begin{aligned}
&n_{a_1}+m_{a_1}=n_{a_2}+m_{a_2},
&s_1n_{a_1}+s_2m_{a_1}=s_1n_{a_2}+s_2m_{a_2}.
\end{aligned}
\end{equation*}
It is easily checked that $(s_1-s_2)(m_{a_1}-m_{a_2})=0$. Since $(s_1,s_2)\neq(1,1)$, we have $n_{a_1}=n_{a_2},m_{a_1}=m_{a_2}$. It implies $F$ is a monomial.
\end{proof}

In the following, we begin to prove Corollary \ref{coro1} and \ref{coro2}.

\begin{proof}[Proof of Corollary \ref{coro1}]
Note that
$$
G(x,y)=s_1xQ(x,y)-s_2yP(x,y)
$$
is a $(s_1,s_2)$-quasi homogeneous polynomial of weight degree $s_1+s_2+m-1$ and has no real linear factors because of Theorem \ref{thm-1}. By Proposition 6 in \cite{FHX}, $s_1+s_2+m-1$ is divisible by $s_1s_2$ and there exists a homogeneous polynomial $G^*(x,y)$ of degree $m^*=\frac{s_1+s_2+m-1}{s_1s_2}$ having no real linear factors and $G(x,y)=G^*(x^{s_2},y^{s_1})$.

We claim $m^*$ is even. If $m^*$ is odd, then $G^*(1,\lambda)=0$ must have a real root $\lambda=\lambda^*\neq 0$. Since $s_1,s_2$ are coprime, either $s_1$ or $s_2$ is odd. Then it is easily checked that $G(x,y)$ has a real zero $(x,y)=(1,(\lambda^*)^{\frac{1}{s_1}})$ if $s_1$ is odd and $G(x,y)$ has a real zero $(x,y)=(\frac{1}{(\lambda^*)^{\frac{1}{s_2}}},1)$ if $s_2$ is odd. It contradicts condition (i) in Theorem \ref{thm-1}. Thus, $s_1+s_2+m-1=m^*s_1s_2$ is even.
\end{proof}

\begin{proof}[Proof of Corollary \ref{coro2}]
We  divide two cases on the relationship between $s_1$ and $s_2$ to prove the conclusion.
\par \textbf{(1)} $s_1=s_2=1$. In this case, $P$ and $Q$ are both homogeneous polynomials of degree $m$. It is known that the infinite equilibria are determined by the real linear factors of homogeneous polynomial $G(x,y)$. And it follows from statement (i) in Theorem \ref{thm-1} that $G(x,y)$ has no real linear factors. This means system \eqref{eq1} has no infinite equilibria.

\par \textbf{(2)} $s_1>s_2$. Suppose $\deg P=p,\ \deg Q=q$. We first show $p>q$ if $s_1>s_2$. We denote the $l$-th homogeneous segment of polynomial $F$ by $F_l$. By Lemma \ref{lem4}, $Q_q$ is a monomial. Then let $Q_q=b_qx^{n_q}y^{m_q}$, where $b_q\neq 0$. Recall $Q_q$ are $(s_1,s_2)$-quasi-homogeneous polynomials of weight degree $m+s_2-1$ respectively by definition of quasi-homogeneous systems. Suppose $p\leq q$.
Note for any monomial $\gamma x^\alpha y^\beta$ in $G(x,y)$,
$$
\alpha+\beta\leq q+1=n_q+m_q+1,\ s_1\alpha+s_2\beta=m+s_1+s_2-1.
$$
Then we have
\begin{equation*}
\begin{aligned}
(s_1-s_2)(\alpha-n_q-1)&\geq s_1(\alpha-n_q-1)+s_2(\beta-m_q)\\
&=(s_1\alpha+s_2\beta)-(s_1(n_q+1)+s_2m_q)\\
&=0,
\end{aligned}
\end{equation*}
that is, $\alpha\geq n_q+1\geq 1$. It implies $G(x,y)$ has the factor $x$, which is a contradiction with statement (i) in Theorem \ref{thm-1}.

It is known that the infinite equilibria are determined by the real linear factors of homogeneous polynomial $-yP_p$ which is equal to $\frac{1}{s_2}G_{p+1}(x,y)$. By Proposition \ref{quasi}, there exists a homogeneous polynomial $G^*(x,y)$ of even degree having no real linear factors and $G(x,y)=G^*(x^{s_2},y^{s_1})$. Hence, it is clear that $G_{p+1}(x,y)$ is a monomial only having variable $y$ when $s_1>s_2$. It follows that system \eqref{eq1} has exactly one pair of infinite equilibria which corresponds to $y=0$ at infinity.
\end{proof}

\section{The Francoise's algorithm for perturbated quasi-homogeneous vector fields}
\label{highOrderMelnikov}
\subsection{Francoise's  algorithm revisited}
Consider perturbated planar Hamiltonian system
\begin{equation}\label{Ham}
dH+\sum_{i=1}^{\infty}\epsilon^i \upsilon_{i}=0,
\end{equation}
where $H$ is an analytic (resp. algebraic) function, $\upsilon_i$ are analytic (resp. algebraic) 1-forms and system $dH=0$ has a center at the origin $O$ and a period annulus $U$.
We say an analytic (resp. algebraic) function $H$ satisfies {\it  Francoise’s condition} (or  called *-property) on $U$ if following two conditions are equivalent for any analytic (resp. algebraic) 1-form $\upsilon_i$ on $U$.\\
\begin{itemize}
\item[(a)] $\oint_{\gamma_h}\upsilon=0$;\\
\item[(b)] there exist analytic (resp. algebraic) functions $a$ and $R$ on $U$ such that
\begin{equation}\label{eqO}
\upsilon=-adH+dR.
\end{equation}
\end{itemize}
In the analytic case, the Francoise's condition always holds (cf. \cite[Corollary 1]{G}). In the algebraic case, Francoise’s condition just holds generically (cf. \cite{Ilyashenko}). An example Francoise’s condition does not hold for the algebraic case is so-called {\it eight-loop Hamiltonian}
\begin{equation}
H(x,y)=\frac{y^2}{2}+\frac{(x^2-1)^2}{4},
\end{equation}
see \cite{JMP, JZ} for details.

Let $\gamma_h$ be an oval of the level set $\{(x,y)|\ H(x,y)=h\}$. Next theorem is proved by \cite{F}, and provides a method to calculate Melnikov function at any order in  the analytic case. It is called {\it Francoise's algorithm}.
\begin{theorem}[Francoise's algorithm]\label{thm-Fran}
Suppose an analytic function $H$ satisfies Francoise's condition on $U$. Then
\begin{itemize}
\item[(1)] $M_1(h)=-\oint_{\gamma_h}\upsilon_{1}$;\\
\item[(2)] if $M_1(h)=\dots =M_{k}(h)\equiv0$($k\geq 1$), then there exist analytic functions $a_0(x,y)\equiv 1,a_1(x,y),\dots, a_{k}(x,y)$ and $R_1(x,y),R_2(x,y)$
$,\dots,R_k(x,y)$ such that
\begin{equation}\label{deq}
\sum_{i=1}^k a_{k-i}\upsilon_{i}=-a_k dH+dR_k,
\end{equation}
and
$$
M_{k+1}(h)=-\oint_{\gamma_h}\sum_{i=1}^{k+1}a_{k+1-i}\upsilon_{i}.
$$
\end{itemize}
\end{theorem}

 It is clear there always exist $a_k(x,y)$ and $R_k(x,y)$ that satisfy the equation \eqref{deq} if Hamiltonian $H$ satisfies Francoise's condition. So the main problem is how to solve $a_k,R_k$ from equation \eqref{deq}.

\subsection{Calculation of Melnikov functions in  quasi-homogeneous vector fields}
In general, the partial differential equation \eqref{deq} is difficult to be solved analytically. In this subsection we first introduce generalized polar coordinates to address the challenging problem by using key properties of quasi-homogeneous analytic functions, and obtain the
adapted Francoise's algorithm to calculate Melnikov function at any order for general quasi-homogeneous polynomial vector fields, then we reduce the estimation of the number of zeros for the Melnikov function to a combinatorial counting task, that is to count  the number of elements within a certain set.

Suppose $H$ is a $(s_1,s_2)$-quasi-homogeneous analytic function and $dH=0$ has a center at the origin  $O$. 
Assume that $\Omega=\sum_{j}\Omega_j$, where $\Omega_j$ is $(s_1,s_2)$-quasi-homogeneous analytic 1-form. If $(a_j(x,y),R_j(x,y))$ is a solution of equation $\Omega_j=-a_jdH+dR_j$, then $a=\sum_j a_j,R=\sum_j R_j$ is a solution of equation \eqref{eqO}. Thus, without loss of generality, let
$$
\Omega=f(x,y)dx+g(x,y)dy
$$
is a $(s_1,s_2)$-quasi-homogeneous analytic 1-form of weight degree $l$. Consider the generalized polar coordinates transformation $\Phi(r,\theta)$
\begin{equation*}
x=r^{s_1}Cs(\theta),\quad y=r^{s_2}Sn(\theta),
\end{equation*}
which is biholomorphic from $(0,+\infty)\times\mathbb{S}^1$ to $\mathbb{R}^2\setminus\{0\}$, it is easily calculated that
\begin{equation}\label{eq-cal}
\Phi^*(\Omega)=\Omega\circ\Phi=r^{l-1}(\bar{f}(\theta)dr+r\bar{g}(\theta)d\theta),
\end{equation}
where
\begin{equation}\label{eq-cal2}
\begin{aligned}
&\bar{f}(\theta)=s_1Cs(\theta) f(Cs(\theta),Sn(\theta))+s_2 Sn(\theta) g(Cs(\theta), Sn(\theta)),\\
&\bar{g}(\theta)=-Sn^{2s_1-1}\theta f(Cs(\theta),Sn(\theta))+ Cs^{2s_2-1}\theta g(Cs(\theta), Sn(\theta)).
\end{aligned}
\end{equation}
Then equation \eqref{eqO} can be transformed into
\begin{equation}\label{eq-tri}
r^{l-1}(\bar{f}(\theta)dr+r\bar{g}(\theta)d\theta)=-\bar{a}(r,\theta)d(r^m\bar{H}(\theta))+d(\bar{R}(r,\theta)),
\end{equation}
where $\bar{f},\bar{g}$ are defined by \eqref{eq-cal2}, $\bar{a}(r,\theta)=\Phi^*(a),\bar{R}(r,\theta)=\Phi^*(R)$ and $\bar{H}(\theta)=H(Cs(\theta),Sn(\theta))$.

 Note that $\bar{H}(\theta)\neq 0$ for any $\theta\in \mathbb{S}^1$, otherwise the line $\theta=\theta_0$ is invariant by the flow of system $dH=0$. Without loss of generality, let $\bar{H}(\theta)>0$ for any $\theta$. Next proposition is easily proved directly by calculation.
\begin{proposition}\label{pro-cal}
The equation \eqref{eq-tri} has a solution
$$
\bar{a}(r,\theta)=r^{l-m}\bar{b}(\theta,\theta_0)=-\frac{1}{m}r^{l-m}(\bar{H}(\theta))^{\frac{l-m}{m}}\int_{\theta_0}^\theta\frac{\bar{f}'(s)-l\bar{g}(s)}{(\bar{H}(s))^{\frac{l}{m}}}ds,
$$
$$
\bar{R}(r,\theta)=\left\{
\begin{aligned}
&\frac{1}{l}r^l(\bar{f}(\theta)-(\bar{H}(\theta))^{\frac{l}{m}}\int_{\theta_0}^\theta\frac{\bar{f}'(s)-l\bar{g}(s)}{(\bar{H}(s))^{\frac{l}{m}}}ds,\quad l\neq 0, \\
&\bar{f}(\theta_0)\ln r+\int_{\theta_0}^\theta(\bar{g}(s)-\frac{\bar{f}(s)-\bar{f}(\theta_0)}{m\bar{H}(s)}\bar{H}'(s))ds,\quad l=0.
\end{aligned}
\right.
$$
\end{proposition}
Proposition \ref{pro-cal} implies that the solution of equation \eqref{eqO} always can be found theoretically by generalized polar coordinates. Thus,  using transformation $\Phi$, we transform system \eqref{Ham} to
\begin{equation}\label{eq-call1}
d(r^m\bar{H}(\theta))+\sum_{i=1}^{\infty}\epsilon^i \Phi^*(\upsilon_{i})=0.
\end{equation}
And the any order  Melnikov functions of \eqref{eq-call1} can be calculated  by Theorem \ref{thm-Fran} and Proposition \ref{pro-cal}.

Consider transformation $\Psi:(r,\theta)\mapsto(\frac{r}{(\bar{H}(\theta))^\frac{1}{m}},\theta)$.  Note this transformation is biholomorphic on $(0,+\infty)\times\mathbb{S}^1$. 1-form \eqref{eq-call1} can be changed into
\begin{equation*}
mr^{m-1}dr+\sum_{i=1}^{\infty}\epsilon^i\Psi^* \Phi^*(\upsilon_{i})=0,
\end{equation*}
which has the same leaves on $(0,+\infty)\times\mathbb{S}^1$ with the foliation associated to following 1-form
\begin{equation}\label{eq-call2}
dr+\sum_{i=1}^{\infty}\epsilon^i\tilde{\upsilon}_i=0,
\end{equation}
where $\tilde{\upsilon}_i=\frac{1}{mr^{m-1}}\Psi^* \Phi^*(\upsilon_{i})$ has the form
$$
\tilde{\upsilon}_i=\sum_{j \in \Lambda_i} r^{j}(\tilde{f}_{i,j}(\theta)dr+r\tilde{g}_{i,j}(\theta)d\theta),
$$
where $\Lambda_i$ is a finite index set that relies on $i$. Using Proposition \ref{pro-cal}, we can get the adapted Francoise's algorithm of 1-form \eqref{eq-call2} under the generalized polar coordinates.
\begin{proposition}\label{thm-Fran2}
\begin{itemize}
\item[(1)] $M_1(h)=-\sum_{j \in \Lambda_1}h^{j+1}\int_0^\tau\tilde{g}_{1,j}(\theta)d\theta$;\\
\item[(2)] If $M_1(h)=\dots =M_{k}(h)\equiv0$ ($k\geq 1$), then there exist $\tau-$periodic functions $a_0(r,\theta)\equiv 1,a_1(r,\theta,\theta_1),a_2(r,\theta,\{\theta_1,\theta_2\}),\dots, a_{k}(r,\theta,\{\theta_n\}_{n=1}^k)$ about variables $\theta,\theta_1,\theta_2,\dots,\theta_k$ such that following conditions hold.\\
\begin{itemize}
\item[(2.a)] $\Omega_{i}=\sum_{j=1}^i a_{i-j}\tilde{\upsilon}_{j}$($i=2,\dots,k+1$) has the form
$$
\Omega_{i}=\sum_{j \in \Lambda_i} r^{j}(\alpha_{i,j}(\theta,\{\theta_n\}_{n=1}^{i-1})dr+r\beta_{i,j}(\theta,\{\theta_n\}_{n=1}^{i-1})d\theta),
$$
where index set $\Lambda_i$ is finite that relies on $i$, and $\alpha_{i,j}$ and $\beta_{i,j}$ are $\tau$-periodic about variables $\theta,\theta_1,\theta_2,\dots,\theta_{i-1}$.\\
\item[(2.b)] $a_i(r,\theta,\{\theta_n\}_{n=1}^i)$ ($i=1,2,\dots,k$) has the form
$$
a_1(r,\theta,\theta_1)=-\sum_{j\in\Lambda_1}r^j\int_{\theta_1}^\theta(\tilde{f}_{1,j}'(s)-(j+1)\tilde{g}_{1,j}(s))ds.
$$
$$
a_i(r,\theta,\{\theta_n\}_{n=1}^i)=-\sum_{j \in \Lambda_i} r^{j}\int_{\theta_i}^\theta(\frac{\partial \alpha_{i,j}}{\partial \theta}(s,\{\theta_n\}_{n=1}^{i-1})-(j+1)\beta_{i,j}(s,\{\theta_n\}_{n=1}^{i-1}))ds,\quad i\geq 2.
$$
Furthermore,
$$
M_{k+1}(h)=-\sum_{j \in \Lambda_{k+1}} h^{j+1}\int_0^\tau\beta_{k+1,j}(\theta,\{\theta_n\}_{n=1}^{k})d\theta.
$$
\end{itemize}
\end{itemize}
\end{proposition}
\begin{proof}
Using Theorem \ref{thm-Fran} and Proposition \ref{pro-cal}, statement (1) and statements (2.a),(2.b) can be checked directly by calculation.

We only need to prove $a_{k}(r,\theta,\{\theta_n\}_{n=1}^k)$ is a $\tau-$periodic function about variables $\theta,\theta_1,\theta_2,\dots,\theta_k$ by induction.

Note that  $f_{1,j},g_{1,j}$ are $\tau$-periodic about variable $\theta$. When $k=1$, to prove that
$$
a_1(r,\theta,\theta_1)=-\sum_{j\in\Lambda_1}r^j\int_{\theta_1}^\theta(\tilde{f}_{1,j}'(s)-(j+1)\tilde{g}_{1,j}(s))ds
$$
 is a $\tau$-periodic function about variables $\theta,\theta_1$, it is enough to show that $\int_{0}^\tau(\tilde{f}_{1,j}'(s)-(j+1)\tilde{g}_{1,j}(s))ds=0$ for any $j\in\Lambda_1$.

 In fact,
$$
M_1(h)=-\sum_{j \in \Lambda_1}h^{j+1}\int_0^\tau\tilde{g}_{1,j}(\theta)d\theta\equiv 0,
$$
which leads to that  $\int_0^\tau\tilde{g}_{1,j}(\theta)d\theta=0$ for any $j\in\Lambda_1$. Thus
\begin{equation*}
\begin{aligned}
&\int_0^\tau (\tilde{f}_{1,j}'(s)-(j+1)\tilde{g}_{1,j}(s))ds\\
=&(\tilde{f}_{1,j}(\tau)-\tilde{f}_{1,j}(0))-(j+1)\int_0^\tau \tilde{g}_{1,j}(s)ds\\
=&0.
\end{aligned}
\end{equation*}
Now assume that $a_{k}(r,\theta,\{\theta_n\}_{n=1}^k)$ is a $\tau-$periodic function about variable $\theta,\theta_1,\theta_2,\dots,\theta_k$ if $k\leq l-1$. When $k=l$, by (2.b),
$$
a_l(r,\theta,\{\theta_n\}_{n=1}^l)=-\sum_{j \in \Lambda_l} r^{j}\int_{\theta_l}^\theta(\frac{\partial \alpha_{l,j}}{\partial \theta}(s,\{\theta_n\}_{n=1}^{l-1})-(j+1)\beta_{l,j}(s,\{\theta_n\}_{n=1}^{l-1}))ds.
$$

By inductive hypotheses it follows that $a_l$ is periodic about variables $\theta_1,\dots,\theta_{l-1}$. To prove $a_l$ is a $\tau$-periodic function about variables $\theta_1$, $\cdots, \theta_l$, it is enough to prove
$$
\int_{0}^\tau(\frac{\partial \alpha_{l,j}}{\partial \theta}(s,\{\theta_n\}_{n=1}^{l-1})-(j+1)\beta_{l,j}(s,\{\theta_n\}_{n=1}^{l-1}))ds=0
$$
for any $j\in\Lambda_l$.

In fact,
$$
M_l(h)=-\oint_{r=h}\Omega_l=-\sum_{j \in \Lambda_1}h^{j+1}\int_0^\tau\beta_{l,j}(\theta,\{\theta_n\}_{n=1}^{l-1})d\theta=0
$$
means $\int_0^\tau\beta_{l,j}(\theta,\{\theta_n\}_{n=1}^{l-1})d\theta=0$ for any $j\in\Lambda_l$. Thus
\begin{equation*}
\begin{aligned}
&\int_{0}^\tau(\frac{\partial \alpha_{l,j}}{\partial \theta}(s,\{\theta_n\}_{n=1}^{l-1})-(j+1)\beta_{l,j}(s,\{\theta_n\}_{n=1}^{l-1}))ds\\
=&(\alpha_{l,j}(\tau,\{\theta_n\}_{n=1}^{l-1})-\alpha_{l,j}(0,\{\theta_n\}_{n=1}^{l-1}))-(j+1)\int_0^\tau \beta_{l,j}(s,\{\theta_n\}_{n=1}^{l-1})ds\\
=&0.
\end{aligned}
\end{equation*}
\end{proof}
Note that $M_k(h)=\frac{d^kP}{d\epsilon^k}(h,\epsilon)$, where $P(h,\epsilon)$ is the Poincar\'e map of system \eqref{eq-call2}. $P(h,\epsilon)$ only depends on the perturbated terms $\tilde{\upsilon}_i$, and is not related to the value of $\theta_1,\theta_2,\dots,\theta_{k-1}$.
\begin{proposition}\label{pro-no}
Melnikov functions $M_k$ do not depend on the values of $\theta_i$, i.e. if $M_1(h)=\dots =M_{k}(h)\equiv0$($k\geq 1$), then
$$
\oint_{\{r=h,\ 0\leq\theta\leq\tau\}}\Omega_{k+1}(r,\theta,\{\theta_{1n}\}_{n=1}^k)=\oint_{\{r=h,\ 0\leq\theta\leq\tau\}}\Omega_{k+1}(r,\theta,\{\theta_{2n}\}_{n=1}^k)
$$
for any $\{\theta_{1n}\}_{n=1}^k$ and $\{\theta_{2n}\}_{n=1}^k$.
\end{proposition}

In what follows, we discuss how to calculate the Melnikov functions and determine the index set $\Lambda_i$ in Proposition \ref{thm-Fran2} for the general quasi-homogeneous polynomial systems. Suppose system \eqref{eq1} is $(s_1,s_2)$-quasi-homogeneous and has a center at the origin $O$. Through the discussion in Section 2, the equation \eqref{eq1} has an integral factor
$$
\mu(x,y)=\frac{1}{s_1xQ(x,y)-s_2yP(x,y)}
$$
defined on $\mathbb{R}^2\setminus\{O\}$ such that $\frac{\partial (\mu P)}{\partial x}+\frac{\partial (\mu Q)}{\partial y}=0$, i.e. the unperturbated system \eqref{eq1} has an analytic first integral $H$ defined on the period annulus $\mathbb{R}^2\setminus\{O\}$ such that $dH=\mu\omega_0$, and $H$ is an analytic quasi-homogeneous function. Thus,  system \eqref{eq22} can be transformed to the perturbated Hamiltonian system \eqref{eq221} (further, system \eqref{eq-call2}) by the integral factor $\mu$, then we can use Proposition \ref{thm-Fran2} to calculate the Melnikov functions of perturbated quasi-homogeneous system \eqref{eq2}.

Specifically, let
$$
\mathcal{T}_{r}^{(s_1,s_2)}=\{is_1+js_2|i,j\in\mathbb{N},i+j\leq r\},
$$
and
$$
\begin{aligned}
\mathcal{S}_{r,t}^{(s_1,s_2,m)}&=\{i-j(s_1+s_2+m-1)|i\in\mathcal{T}_{j\cdot r}^{(s_1,s_2)}\setminus\mathcal{T}_{j-1}^{(s_1,s_2)},j=1,2,\dots,t\}\\
&=\{is_1+js_2-l(s_1+s_2+m-1)|i,j,l\in\mathbb{N},l\leq i+j\leq l\cdot r,1\leq l\leq t\},
\end{aligned}
$$
where $r$ and $t$ are both positive integers. Note that 1-form $x^iy^jdx$ and $x^iy^jdy$ are $(s_1,s_2)$-quasi-homogeneous 1-form of weight degree $(i+1)s_1+js_2$ and $is_1+(j+1)s_2$ respectively, then we have that 1-form $\omega_i$ in equation \eqref{eq22} has following decomposition
$$
\omega_i=Q_i(x,y)dx-P_i(x,y)dy=\sum_{j\in \mathcal{T}_{n+1}^{(s_1,s_2)}\setminus\{0\}}\omega_{i,j},
$$
where $\omega_{i,j}=Q_{i,j}(x,y)dx-P_{i,j}(x,y)dy$ is a $(s_1,s_2)$-quasi-homogeneous polynomial 1-form of weight degree $j$, that is, $Q_{i,j}$ (resp. $P_{i,j}$) is the $(s_1,s_2)$-quasi-homogeneous segment of weight degree $j-s_1$ (resp. $j-s_2$) of $Q_i$ (resp. $P_i$). The perturbated differential equation \eqref{eq22} can be transformed by the generalized polar coordinates transformation $\Phi$ into
\begin{equation}\label{eq23}
r^{s_1+s_2+m-2}(G(\theta)dr-rF(\theta)d\theta)+\sum_{i=1}^\infty \epsilon^i\sum_{j\in \mathcal{T}_{n+1}^{(s_1,s_2)}\setminus\{0\}}\bar\omega_{i,j}=0,
\end{equation}
where $F,G$ is defined by equality \eqref{eq-fg}, and
$$
\bar{\omega}_{i,j}=\Phi^*(\omega_{i,j})=r^{j-1}(G_{i,j}(\theta)dr-rF_{i,j}(\theta)d\theta),
$$
where
\begin{equation}\label{eq-FGij}
\begin{aligned}
&F_{i,j}(\theta)=Cs^{2s_2-1}\theta P_{i,j}(Cs(\theta), Sn(\theta))+Sn^{2s_1-1}\theta Q_{i,j}(Cs(\theta),Sn(\theta)),\\
&G_{i,j}(\theta)=s_1 Cs(\theta) Q_{i,j}(Cs(\theta),Sn(\theta))-s_2 Sn(\theta) P_{i,j}(Cs(\theta),Sn(\theta)).
\end{aligned}
\end{equation}
By Theorem \ref{thm-1}, $G(\theta)\neq 0$ for any $\theta\in\mathbb{S}^1$ and $r>0$.

Let $H(\theta)=e^{\int_{0}^\theta\frac{F(s)}{G(s)}ds}$. Both sides of equation \eqref{eq23} are multiplied by $\frac{1}{r^{s_1+s_2+m-2}G(\theta)H(\theta)}$ and we consider transformation $r\mapsto rH(\theta),\theta\mapsto\theta$. Then equation \eqref{eq23} is transformed to a similar form with \eqref{eq-call2}
\begin{equation}\label{eq24}
dr+\sum_{i=1}^\infty \epsilon^i\tilde{\omega}_i=0,
\end{equation}
where
$$
\tilde{\omega}_i=\sum_{j\in \mathcal{S}_{n+1,1}^{(s_1,s_2,m)}}\tilde{\omega}_{i,j}=\sum_{j\in \mathcal{S}_{n+1,1}^{(s_1,s_2,m)}}r^{j}(g_{i,j}(\theta)dr-rf_{i,j}(\theta)d\theta),
$$
and
\begin{equation}\label{eq-fgij}
\begin{aligned}
&f_{i,j}(\theta)=(H(\theta))^{j}\frac{F_{i,j+s_1+s_2+m-1}(\theta)G(\theta)-G_{i,j+s_1+s_2+m-1}(\theta)F(\theta)}{(G(\theta))^2},\\
&g_{i,j}(\theta)=(H(\theta))^{j}\frac{G_{i,j+s_1+s_2+m-1}(\theta)}{G(\theta)}.
\end{aligned}
\end{equation}

\begin{theorem}[Francoise's algorithm for perturbated quasi-homogeneous polynomial vector fields]
\label{thm-Fran3}
In generalized polar coordinates, the foliation $\mathcal{F}(\omega_\epsilon)$ takes the form \eqref{eq24}, and its  Melnikov functions are determined as follows
\begin{itemize}
\item[(i)] If $M_1(h)\not\equiv 0$, then
$$
\begin{aligned}
M_{1}(h)=&\sum_{j \in \mathcal{S}_{n+1,1}^{(s_1,s_2,m)}} h^{j+1}m_{1,j}\\
=&\sum_{j \in \mathcal{S}_{n+1,1}^{(s_1,s_2,m)}}h^{j+1}\int_0^\tau f_{1,j}(\theta)d\theta.
\end{aligned}
$$\\
\item[(ii)] If $M_1(h)=\dots =M_{k}(h)\equiv0$($k\geq 1$), then there exist $\tau-$periodic functions $a_0(r,\theta)\equiv 1,a_1(r,\theta,\theta_1),a_2(r,\theta,\{\theta_1,\theta_2\}),\dots, a_{k}(r,\theta,\{\theta_n\}_{n=1}^k)$ about variables $\theta,\theta_1,\theta_2,\dots,\theta_k$ such that following conditions hold.
\begin{itemize}
\item[(ii.a)] $\Omega_{l}=\sum_{j=1}^l a_{l-j}\tilde{\omega}_{j}$($l=2,\dots,k+1$) has the form
$$
\Omega_{l}=\sum_{j \in \mathcal{S}_{n+1,l}^{(s_1,s_2,m)}} r^{j}(A_{l,j}(\theta,\{\theta_n\}_{n=1}^{l-1})dr-rB_{l,j}(\theta,\{\theta_n\}_{n=1}^{l-1})d\theta),
$$
where $A_{j}$ and $B_{j}$ are $\tau$-periodic about variables $\theta,\theta_1,\theta_2,\dots,\theta_{i-1}$.\\
\item[(ii.b)] $a_i(r,\theta,\{\theta_n\}_{n=1}^i)$ ($i=1,2,\dots,k$) has the form
$$
\begin{aligned}
a_1(r,\theta,\theta_1)&=-\sum_{j\in\mathcal{S}_{n+1,1}^{(s_1,s_2,m)}}r^{j}b_{1,j}(\theta,\theta_1)\\
&=-\sum_{j\in\mathcal{S}_{n+1,1}^{(s_1,s_2,m)}}r^{j}(g_{1,j}(\theta)+(j+1)\int_{\theta_1}^\theta f_{1,j}(s)ds),
\end{aligned}
$$
$$
\begin{aligned}
a_i(r,\theta,\{\theta_n\}_{n=1}^i)=&-\sum_{j \in \mathcal{S}_{n+1,i}^{(s_1,s_2,m)}} r^{j}b_{i,j}(\theta,\{\theta_n\}_{n=1}^i)\\
=&-\sum_{j \in \mathcal{S}_{n+1,i}^{(s_1,s_2,m)}} r^{j}(A_{i,j}(\theta,\{\theta_n\}_{n=1}^{i-1})+(j+1)\int_{\theta_i}^\theta B_{i,j}(s,\{\theta_n\}_{n=1}^{i-1})ds),\quad i\geq 2.
\end{aligned}
$$
\end{itemize}
Furthermore,
$$
\begin{aligned}
M_{k+1}(h)=&\sum_{j \in \mathcal{S}_{n+1,k+1}^{(s_1,s_2,m)}} h^{j+1}m_{k+1,j}\\
=&\sum_{j \in \mathcal{S}_{n+1,k+1}^{(s_1,s_2,m)}} h^{j+1}\int_0^\tau B_{k+1,j}(\theta,\{\theta_n\}_{n=1}^{k})d\theta.
\end{aligned}
$$
\end{itemize}
\end{theorem}

\begin{proof}
The conclusion (i) is clear. We just prove statement (ii). The form and periodicity of $a_i$s have been proved in Proposition \ref{thm-Fran2}. We only need to check the elements in the index set $\Lambda_i$ in Proposition \ref{thm-Fran2} by induction. Suppose statement (ii) holds. If $M_1(h)=\dots=M_k(h)\equiv 0$, then when $1\leq i\leq k$, we have
\begin{equation}\label{eq-aomega}
\begin{aligned}
a_{k+1-i}\tilde{\omega}_{i}&=-\sum_{j \in \mathcal{S}_{n+1,k+1-i}^{(s_1,s_2,m)}} r^{j}b_{i,j}\cdot\sum_{l\in \mathcal{S}_{n+1,1}^{(s_1,s_2,m)}}r^{l}(g_{i,l}(\theta)dr-rf_{i,l}(\theta)d\theta)\\
&=-\sum_{(j,l) \in \mathcal{S}_{n+1,k+1-i}^{(s_1,s_2,m)}\times\mathcal{S}_{n+1,1}^{(s_1,s_2,m)}}r^{j+l}(b_{i,j}g_{i,l}(\theta)dr-rb_{i,j}f_{i,l}(\theta)d\theta)\\
&=-\sum_{s \in {\Gamma_i}}r^{s}(\sum_{j+l=s}b_{i,j}g_{i,l}(\theta)dr-r\sum_{j+l=s}b_{i,j}f_{i,l}(\theta)d\theta),
\end{aligned}
\end{equation}
where
$$
\begin{aligned}
\Gamma_i&=\{p-q(s_1+s_2+m-1)|p\in\mathcal{T}_{q\cdot (n+1)}^{(s_1,s_2)}\setminus\mathcal{T}_{q-1}^{(s_1,s_2)},q=2,\dots,k+2-i\}\\
&=\bigcup_{q=2}^{k+2-i}\{p-q(s_1+s_2+m-1)|p\in\mathcal{T}_{q\cdot (n+1)}^{(s_1,s_2)}\setminus\mathcal{T}_{q-1}^{(s_1,s_2)}\}.\\
\end{aligned}
$$
Note that  $\Omega_{k+1}=\sum_{i=1}^{k}a_{k+1-i}\tilde{\omega}_{i}+\omega_{k+1}$. Hence, the index set $\Lambda_{k+1}$ has the form
$$
\begin{aligned}
\Lambda_{k+1}&=(\bigcup_{i=1}^{k}\Gamma_i)\bigcup (\mathcal{S}_{n+1,1}^{(s_1,s_2,m)}) \\
&=(\bigcup_{i=1}^{k}\bigcup_{q=2}^{k+2-i}\{p-q(s_1+s_2+m-1)|p\in\mathcal{T}_{q\cdot (n+1)}^{(s_1,s_2)}\setminus\mathcal{T}_{q-1}^{(s_1,s_2)}\})\bigcup (\mathcal{S}_{n+1,1}^{(s_1,s_2,m)})\\
&=\bigcup_{q=1}^{k+1}\{p-q(s_1+s_2+m-1)|p\in\mathcal{T}_{q\cdot (n+1)}^{(s_1,s_2)}\setminus\mathcal{T}_{q-1}^{(s_1,s_2)}\}\\
&=\mathcal{S}_{n+1,k+1}^{(s_1,s_2,m)}.
\end{aligned}
$$
\end{proof}

\section{The cyclicity on period annulus of quasi-homogeneous polynomial vector fields}\label{AHproblem}
In this section, we consider the perturbated system \eqref{eq2} and estimate its zeros of the first non-vanishing Melnikov function of order $k$. By Theorem \ref{thm-Fran3}, it is clear that the number of positive zeros of $M_k(h)$ is less than the cardinal of the set $\mathcal{S}_{n+1,k+1}^{(s_1,s_2,m)}$ (denoted by $\sharp\mathcal{S}_{n+1,k+1}^{(s_1,s_2,m)}$). In fact, we can reduce the number of zeros by almost half. Let
$$
\tilde{\mathcal{S}}_{n+1,k}^{(s_1,s_2,m)}=\left\{
\begin{aligned}
&\mathcal{S}_{n+1,k}^{(s_1,s_2,m)}\cap(2\mathbb{Z}),\quad \text{ if $s_1,s_2$ are both odd},\\
&\mathcal{S}_{n+1,k}^{(s_1,s_2,m)}-(2\mathbb{Z}),\quad \text{ if $s_1+s_2$ is odd,}
\end{aligned}
\right.
$$
where $\mathbb{Z}$ is the set of integers.
Our main result is as follow.

\begin{theorem}\label{mainthm}
Assume that system \eqref{eq1} is $(s_1,s_2)$-quasi-homogeneous with weight degree $m$, and $P_i(x,y)$ and $Q_i(x,y)$ in system \eqref{eq2} are polynomials of degree not greater than $n$. If system \eqref{eq1} has a center at the origin $O$, then the first non-vanishing $k$th order Melnikov function $M_k(h)$ of system \eqref{eq2} has at most $\sharp\tilde{\mathcal{S}}_{n+1,k}^{(s_1,s_2,m)}-1$ zeros on $(0,+\infty)$, where
$$
\tilde{\mathcal{S}}_{n+1,k}^{(s_1,s_2,m)}=\left\{
\begin{aligned}
&\mathcal{S}_{n+1,k}^{(s_1,s_2,m)}\cap(2\mathbb{Z}),\quad \text{ if $s_1,s_2$ are both odd},\\
&\mathcal{S}_{n+1,k}^{(s_1,s_2,m)}-(2\mathbb{Z}),\quad \text{ if $s_1+s_2$ is odd,}
\end{aligned}
\right.
$$
and
$$
\mathcal{S}_{n+1,k}^{(s_1,s_2,m)}=\{is_1+js_2-l(s_1+s_2+m-1)|i,j,l\in\mathbb{N},l\leq i+j\leq l\cdot (n+1),1\leq l\leq k\}.
$$
\end{theorem}


\begin{proof}
In fact, Theorem \ref{thm-Fran3} shows that the first non-vanishing Melnikov function of order $k$ has the form
$$
M_k(h)=\sum_{j \in \mathcal{S}_{n+1,k}^{(s_1,s_2,m)}} h^{j+1}m_{k,j}.
$$
By Descartes Theorem (cf. \cite{BZ}), $M_k(h)$ has at most
$$
\sharp\mathcal{S}_{n+1,k}^{(s_1,s_2,m)}\setminus \{j|m_{k,j}= 0\}-1
$$
positive zeros. We claim

\begin{lemma}\label{lem-sectionAH-1}
$m_{k,j}=0$ if one of following conditions holds.
\begin{itemize}
\item[(a)] $s_1,s_2$ and $j$ are all odd;
\item[(b)] $s_1+s_2$ is odd and $j$ is even.
\end{itemize}
\end{lemma}

By Lemma \ref{lem-sectionAH-1} and Corollary \ref{coro1}, we have
$$
\mathcal{S}_{n+1,k}^{(s_1,s_2,m)}\setminus \{j|m_{k,j}= 0\}\subseteq\tilde{\mathcal{S}}_{n+1,k}^{(s_1,s_2,m)},
$$
which implies that the number of zeros of $M_k$ can be bounded by $\sharp\tilde{\mathcal{S}}_{n+1,k}^{(s_1,s_2,m)}-1$.
\end{proof}

The remaining question is to prove Lemma \ref{lem-sectionAH-1}. We first discuss the properties of function $f_{i,j}(\theta),g_{i,j}(\theta)$ and $H(\theta)$ defined in  \eqref{eq24} and $A_{k+1,j},B_{k+1,j},b_{k,j}$ defined in Theorem \ref{thm-Fran3}, then finish the proof of Lemma \ref{lem-sectionAH-1}.
\begin{proposition}\label{pro-ind}
(I). If $s_1,s_2$ are both odd, then
 $$
\quad f_{i,j}(\frac{\tau}{2}+\theta)=(-1)^jf_{i,j}(\theta),\quad g_{i,j}(\frac{\tau}{2}+\theta)=(-1)^jg_{i,j}(\theta).
 $$
 (II). If $s_1$ is odd and $s_2$ is even, then
 $$
\quad f_{i,j}(\frac{\tau}{2}-\theta)=(-1)^{j+1}f_{i,j}(\theta),\quad g_{i,j}(\frac{\tau}{2}-\theta)=(-1)^jg_{i,j}(\theta).
 $$
 (III). If $s_1$ is even and $s_2$ is odd, then
 $$
f_{i,j}(\tau-\theta)=(-1)^{j+1}f_{i,j}(\theta),\quad g_{i,j}(\tau-\theta)=(-1)^jg_{i,j}(\theta).
 $$
\end{proposition}
\begin{proof}
We only prove statement (I) since the proofs of conclusions (II) and (III) are similar. Suppose $s_1,s_2$ are both odd, then $m$ is odd too because of Corollary \ref{coro1}. By Proposition \ref{tri-pro}, we have
\begin{equation*}
\begin{aligned}
G(\theta+\frac{\tau}{2})&=s_1 (-Cs(\theta)) Q(-Cs(\theta),-Sn(\theta))-s_2 (-Sn(\theta)) P(-Cs(\theta),-Sn(\theta)),\\
&=(-1)^{s_2+m}s_1 Cs(\theta) Q(Cs(\theta),Sn(\theta))-(-1)^{s_1+m}s_2 Sn(\theta) P(Cs(\theta),Sn(\theta))\\
&=G(\theta).
\end{aligned}
\end{equation*}
Similarly, we can prove
$$
F(\theta+\frac{\tau}{2})=F(\theta),\quad F_{i,j}(\theta+\frac{\tau}{2})=(-1)^jF_{i,j}(\theta),\quad G_{i,j}(\theta+\frac{\tau}{2})=(-1)^jG_{i,j}(\theta)
$$
via  \eqref{eq-FGij}.
Then we have
\begin{equation*}
\begin{aligned}
\int_{-\frac{\tau}{2}}^{\frac{\tau}{2}}\frac{F(s)}{G(s)}ds&=(\int_{-\frac{\tau}{2}}^0+\int_0^\frac{\tau}{2})\frac{F(s)}{G(s)}ds,\\
&=\int_{-\frac{\tau}{2}}^0\frac{F(s)}{G(s)}ds+\int_{-\frac{\tau}{2}}^0\frac{F(s+\frac{\tau}{2})}{G(s+\frac{\tau}{2})}ds,\\
&=2\int_{-\frac{\tau}{2}}^0\frac{F(s)}{G(s)}ds.
\end{aligned}
\end{equation*}
By Theorem \ref{thm-1}, we have
$$
\int_{-\frac{\tau}{2}}^0\frac{F(s)}{G(s)}ds=\frac{1}{2}\int_{-\frac{\tau}{2}}^{\frac{\tau}{2}}\frac{F(s)}{G(s)}ds=0.
$$
Hence,
\begin{equation*}
\begin{aligned}
H(\theta+\frac{\tau}{2})&=e^{\int_{0}^{\theta+\frac{\tau}{2}}\frac{F(s)}{G(s)}ds}
=e^{\int_{-\frac{\tau}{2}}^{\theta}\frac{F(s+\frac{\tau}{2})}{G(s+\frac{\tau}{2})}ds}\\
&=e^{(\int_{-\frac{\tau}{2}}^0+\int_0^{\theta})\frac{F(s)}{G(s)}ds}
=e^{\int_0^{\theta}\frac{F(s)}{G(s)}ds}=H(\theta).
\end{aligned}
\end{equation*}
Using a similar approach, it follows directly from  \eqref{eq-fgij} that
 $$
 f_{i,j}(\frac{\tau}{2}+\theta)=(-1)^jf_{i,j}(\theta),\quad g_{i,j}(\frac{\tau}{2}+\theta)=(-1)^jg_{i,j}(\theta).
 $$
\end{proof}
It is enough to calculate the Melnikov functions of system \eqref{eq24} by Proposition \ref{pro-cal} and Theorem \ref{thm-Fran2}. Let $\Omega_k=\sum_{i=1}^k \bar{a}_{k-i}(r,\theta,\theta_{k-i})\tilde{\omega}_{i}$, where $\bar{a}$ is calculated by Francoise's algorithm and Proposition \ref{pro-cal} for system \eqref{eq24}. Firstly, we show:
\begin{lemma}\label{mlemma}
Functions $A_{k+1,j},B_{k+1,j}$ and $b_{k,j}$ are given in Theorem \ref{thm-Fran3}, then the following statements hold.
\begin{itemize}
\item[(1)] If $s_1,s_2$ are both odd, then
$$
A_{k+1,j}(\theta+\frac{\tau}{2},\{\theta_n\}_{n=1}^{k})=(-1)^jA_{k+1,j}(\theta,\{\theta_n-\frac{\tau}{2}\}_{n=1}^{k}),
$$
$$
B_{k+1,j}(\theta+\frac{\tau}{2},\{\theta_n\}_{n=1}^{k})=(-1)^jB_{k+1,j}(\theta,\{\theta_n-\frac{\tau}{2}\}_{n=1}^{k}),
$$
$$
b_{k,j}(\theta+\frac{\tau}{2},\{\theta_n\}_{n=1}^{k})=(-1)^jb_{k,j}(\theta,\{\theta_n-\frac{\tau}{2}\}_{n=1}^{k}).
$$
\item[(2)] If $s_1$ is odd and $s_2$ is even, then
$$
A_{k+1,j}(\frac{\tau}{2}-\theta,\{\theta_n\}_{n=1}^{k})=(-1)^jA_{k+1,j}(\theta,\{\frac{\tau}{2}-\theta_n\}_{n=1}^{k}),
$$
$$
B_{k+1,j}(\frac{\tau}{2}-\theta,\{\theta_n\}_{i=1}^{k})=(-1)^{j+1}B_{k+1,j}(\theta,\{\frac{\tau}{2}-\theta_n\}_{n=1}^{k}),
$$
$$
b_{k,j}(\frac{\tau}{2}-\theta,\{\theta_n\}_{n=1}^{k})=(-1)^jb_{k,j}(\theta,\{\frac{\tau}{2}-\theta_n\}_{n=1}^{k}).
$$
\item[(3)] If $s_1$ is even and $s_2$ is odd, then
$$
A_{k+1,j}(\tau-\theta,\{\theta_n\}_{n=1}^{k})=(-1)^jA_{k+1,j}(\theta,\{\tau-\theta_n\}_{n=1}^{k}),
$$
$$
B_{k+1,j}(\tau-\theta,\{\theta_n\}_{n=1}^{k})=(-1)^{j+1}B_{k+1,j}(\theta,\{\tau-\theta_n\}_{n=1}^{k}),
$$
$$
b_{k,j}(\tau-\theta,\{\theta_n\}_{n=1}^{k})=(-1)^jb_{k,j}(\theta,\{\tau-\theta_n\}_{n=1}^{k}).
$$
\end{itemize}
\end{lemma}
\begin{proof}
To save space, we only prove statement (1) by induction since the proofs of statement (2) and (3) are similar. When $k=1$, by Theorem \ref{thm-Fran3}, we know that $a_1$ has the form
$$
\begin{aligned}
a_1(r,\theta,\theta_1)&=-\sum_{j\in\mathcal{S}_{n+1,1}^{(s_1,s_2,m)}}r^{j}b_{1,j}(\theta,\theta_1)\\
&=-\sum_{j\in\mathcal{S}_{n+1,1}^{(s_1,s_2,m)}}r^{j}(g_{1,j}(\theta)+(j+1)\int_{\theta_1}^\theta f_{1,j}(s)ds).
\end{aligned}
$$
By Francoise's algorithm,
\begin{equation*}
\begin{aligned}
\Omega_{2}&=\bar{a}_1(r,\theta,\theta_1)\tilde{\omega}_1+\tilde{\omega}_2\\
&=-\sum_{i\in\mathcal{S}_{n+1,1}^{(s_1,s_2,m)}}r^{i}b_{1,i}(\theta,\theta_1)\cdot\sum_{j\in \mathcal{S}_{n+1,1}^{(s_1,s_2,m)}}r^{j}(g_{1,j}(\theta)dr-rf_{1,j}(\theta)d\theta)+\\
&\sum_{j\in \mathcal{S}_{n+1,1}^{(s_1,s_2,m)}}r^{j}(g_{2,j}(\theta)dr-rf_{2,j}(\theta)d\theta)\\
&=\sum_{j\in\mathcal{S}_{n+1,2}^{(s_1,s_2,m)}}r^{j}((\sum_{i_1+i_2=j}b_{1,i_1}g_{1,i_2}+g_{2,j})dr-r(\sum_{i_1+i_2=j}b_{1,i_1}f_{1,i_2}+f_{2,j})d\theta).
\end{aligned}
\end{equation*}
It follows that
$$
A_{2,j}=\sum_{i_1+i_2=j}b_{1,i_1}(\theta,\theta_1)g_{1,i_2}(\theta)+g_{2,j}(\theta),\quad B_{2,j}=\sum_{i_1+i_2=j}b_{1,i_1}(\theta,\theta_1)f_{1,i_2}(\theta)+f_{2,j}(\theta).
$$
By Proposition \ref{pro-ind}, it is clear that
\begin{equation}\label{e1}
\begin{aligned}
b_{1,j}(\theta+\frac{\tau}{2},\theta_1)&=g_{1,j}(\theta+\frac{\tau}{2})+(j+1)\int_{\theta_1}^{\theta+\frac{\tau}{2}} f_{1,j}(s)ds\\
&=(-1)^jg_{1,j}(\theta)+(j+1)\int_{\theta_1-\frac{\tau}{2}}^\theta f_{1,j}(s+\frac{\tau}{2})d(s+\frac{\tau}{2})\\
&=(-1)^jb_{1,j}(\theta,\theta_1-\frac{\tau}{2}).
\end{aligned}
\end{equation}
Thus, by Proposition \ref{pro-ind} and  \eqref{e1}, we have
$$
\begin{aligned}
A_{2,j}(\theta+\frac{\tau}{2},\theta_1)&=\sum_{i_1+i_2=j}b_{1,i_1}(\theta+\frac{\tau}{2},\theta_1)g_{1,i_2}(\theta+\frac{\tau}{2})+g_{2,j}(\theta+\frac{\tau}{2})\\
&=(-1)^j(\sum_{i_1+i_2=j}b_{1,i_1}(\theta,\theta_1-\frac{\tau}{2})g_{1,i_2}(\theta)+g_{2,j}(\theta))\\
&=(-1)^jA_{2,j}(\theta,\theta_1-\frac{\tau}{2}),
\end{aligned}
$$
and by a similar approach, $B_{2,j}(\theta+\frac{\tau}{2},\theta_1)=(-1)^jB_{2,j}(\theta,\theta_1-\frac{\tau}{2})$.

Assume that the conclusion is true as $k\leq l-1$. When $k=l$, it follows from induction that
\begin{equation}\label{e2}
\begin{aligned}
b_{l,j}(\theta+\frac{\tau}{2},\{\theta_n\}_{n=1}^{l})&=A_{l,j}(\theta+\frac{\tau}{2},\{\theta_n\}_{n=1}^{l-1})+(j+1)\int_{\theta_l}^{\theta+\frac{\tau}{2}} B_{l,j}(s,\{\theta_n\}_{n=1}^{l-1})ds\\
&=(-1)^jA_{l,j}(\theta,\{\theta_n-\frac{\tau}{2}\}_{n=1}^{l-1})+(j+1)\int_{\theta_l-\frac{\tau}{2}}^{\theta} B_{l,j}(s+\frac{\tau}{2},\{\theta_n\}_{n=1}^{l-1})d(s+\frac{\tau}{2})\\
&=(-1)^j(A_{l,j}(\theta,\{\theta_n-\frac{\tau}{2}\}_{n=1}^{l-1})+(j+1)\int_{\theta_l-\frac{\tau}{2}}^{\theta} B_{l,j}(s,\{\theta_n\}_{n=1}^{l-1})ds)\\
&=(-1)^jb_{l,j}(\theta,\{\theta_n-\frac{\tau}{2}\}_{n=1}^{l}).
\end{aligned}
\end{equation}
From \eqref{eq-aomega}, we have
$$
\begin{aligned}
\Omega_{l+1}&=\sum_{j=1}^{l} a_{l+1-j}\tilde{\omega}_{j}+\tilde{\omega}_{l+1}\\
&=\sum_{j\in\mathcal{S}_{n+1,2}^{(s_1,s_2,m)}}r^{j}((\sum_{i=1}^l\sum_{i_1+i_2=j}b_{i,i_1}g_{i,i_2}+g_{l+1,j})dr-r(\sum_{i=1}^l\sum_{i_1+i_2=j}b_{i,i_1}f_{i,i_2}+f_{l+1,j})d\theta),
\end{aligned}
$$
that is to say,
$$
A_{l+1,j}=\sum_{i=1}^l\sum_{i_1+i_2=j}b_{i,i_1}g_{i,i_2}+g_{l+1,j},\quad B_{l+1,j}=\sum_{i=1}^l\sum_{i_1+i_2=j}b_{i,i_1}f_{i,i_2}+f_{l+1,j},
$$
By Proposition \ref{pro-ind}, equality \eqref{e2} and inductive hypothesis, it follows that
$$
\begin{aligned}
A_{l+1,j}(\theta+\frac{\tau}{2},\{\theta_n\}_{n=1}^{l})&=\sum_{i=1}^l\sum_{i_1+i_2=j}b_{i,i_1}(\theta+\frac{\tau}{2},\{\theta_n\}_{n=1}^{l})g_{i,i_2}(\theta+\frac{\tau}{2})+g_{l+1,j}(\theta+\frac{\tau}{2})\\
&=(-1)^j\sum_{i=1}^l\sum_{i_1+i_2=j}b_{i,i_1}(\theta,\{\theta_n-\frac{\tau}{2}\}_{n=1}^{l})g_{i,i_2}(\theta)+g_{l+1,j}(\theta)\\
&=(-1)^jA_{l+1,j}(\theta,\{\theta_n-\frac{\tau}{2}\}_{n=1}^{l}),
\end{aligned}
$$
and by a similar approach, we obtain
$$
B_{l+1,j}(\theta+\frac{\tau}{2},\{\theta_n\}_{n=1}^{l})=(-1)^jB_{l+1,j}(\theta,\{\theta_n-\frac{\tau}{2}\}_{n=1}^{l}).
$$
\end{proof}

\begin{proof}[Proof of Lemma \ref{lem-sectionAH-1}]
It suffices to verify  $m_{k,j}=0$ if $s_1,s_2$ and $j$ are all odd, the remaining cases follow analogously.
When $k=1$, by Proposition \ref{pro-ind}, we have
$$
\begin{aligned}
m_{1,j}&=(\int_0^\frac{\tau}{2}+\int_\frac{\tau}{2}^\tau) f_{1,j}(\theta)d\theta\\
&=\int_0^{\frac{\tau}{2}}f_{1,j}(\theta)d\theta+\int_0^{\frac{\tau}{2}}f_{1,j}(\theta+\frac{\tau}{2})d(\theta+\frac{\tau}{2})\\
&=(1+(-1)^j)\int_0^{\frac{\tau}{2}}f_{1,j}(\theta)d\theta=0,
\end{aligned}
$$
if $j$ is odd.

When $k\geq 2$,  the value of $\theta_n$ has no effort on the calculation of Melnikov function $M_k(h)$ (Proposition \ref{pro-no}). Hence,
\begin{equation}\label{eq-mkj}
m_{k,j}=\int_0^\tau B_{k,j}(\theta,\{\theta_n\}_{n=1}^{k-1})d\theta=\int_0^\tau B_{k,j}(\theta,\{\theta_n-\frac{\tau}{2}\}_{n=1}^{k-1})d\theta,\quad k\geq 2.
\end{equation}
On the one hand, by Proposition \ref{mlemma}, we have
\begin{equation*}
\begin{aligned}
\int_0^\tau B_{k,j}(\theta,\{\theta_n\}_{n=1}^{k-1})d\theta&=(\int_0^\frac{\tau}{2}+\int_\frac{\tau}{2}^\tau) B_{k,j}(\theta,\{\theta_n\}_{n=1}^{k-1})d\theta\\
=&\int_0^\frac{\tau}{2}B_{k,j}(\theta,\{\theta_n\}_{n=1}^{k-1})d\theta+(-1)^{j}\int_0^\frac{\tau}{2} B_{k,j}(\theta,\{\theta_n-\frac{\tau}{2}\}_{n=1}^{k-1})d\theta.
\end{aligned}
\end{equation*}
On the other hand, note that  $B_{k,j}$ is $\tau-$periodic about variables $\theta_1,\dots,\theta_n$ (Theorem \ref{thm-Fran3}). By Proposition \ref{mlemma}, we have
\begin{equation*}
\begin{aligned}
&\int_0^\tau B_{k,j}(\theta,\{\theta_n-\frac{\tau}{2}\}_{n=1}^{k-1})d\theta\\
=&(\int_0^\frac{\tau}{2}+\int_\frac{\tau}{2}^\tau) B_{k,j}(\theta,\{\theta_n-\frac{\tau}{2}\}_{n=1}^{k-1})d\theta\\
=&\int_0^\frac{\tau}{2}B_{k,j}(\theta,\{\theta_n-\frac{\tau}{2}\}_{n=1}^{k-1})d\theta+(-1)^{j}\int_0^\frac{\tau}{2} B_{k,j}(\theta,\{\theta_n-\tau\}_{n=1}^{k-1})d\theta\\
=&\int_0^\frac{\tau}{2}B_{k,j}(\theta,\{\theta_n-\frac{\tau}{2}\}_{n=1}^{k-1})d\theta+(-1)^{j}\int_0^\frac{\tau}{2} B_{k,j}(\theta,\{\theta_n\}_{n=1}^{k-1})d\theta.
\end{aligned}
\end{equation*}
Thus, when $j$ is odd, it follows from equality \eqref{eq-mkj} that
$$
\int_0^\frac{\tau}{2}B_{k,j}(\theta,\{\theta_n\}_{n=1}^{k-1})d\theta=\int_0^\frac{\tau}{2} B_{k,j}(\theta,\{\theta_n-\frac{\tau}{2}\}_{n=1}^{k-1})d\theta.
$$
Hence,
$$
m_{k,j}=\int_0^\frac{\tau}{2}B_{k,j}(\theta,\{\theta_n\}_{n=1}^{k-1})d\theta+(-1)^{j}\int_0^\frac{\tau}{2} B_{k,j}(\theta,\{\theta_n-\frac{\tau}{2}\}_{n=1}^{k-1})d\theta=0
$$
if $j$ is odd.
\end{proof}

The next challenge  is to estimate the cardinal of the set $\tilde{\mathcal{S}}_{n+1,k}^{(s_1,s_2,m)}$. In the following we
show how to calculate the cardinal in some special cases, and in subsection 4.3 we provide a upper bound of the cardinal of the set $\tilde{\mathcal{S}}_{n+1,k}^{(s_1,s_2,m)}$.

\subsection{Perturbated homogeneous polynomial vector fields}
Let us first consider the homogenous case, that is, $s_1=s_2=1$. In fact, up to now, we do not see the upper bound about the number of zeros of Melnikov functions with any order for the general integrable homogeneous polynomial vector fields. Some special cases have been considered, see \cite{BGL,Iliev}, and for homogeneous Hamiltonian cases, a upper bound of zeros of Melnikov functions with any order is given by \cite{FHX}. Next result extends Theorem 15 in \cite{FHX}.
\begin{corollary}\label{ZeroHomo}
Assume that the degree $m$ homogeneous system \eqref{eq1}  has a center at the origin $O$, the perturbated system \eqref{eq2} is of degree $n$, and $n\geq m$. Then the first non-zero Melnikov function $M_k(h)$ has at most
\begin{equation*}
K(k,n)=\left\{
\begin{aligned}
&[\frac{kn-1}{2}],\quad \text{ $k$ is odd,} \\
&\frac{kn}{2},\quad \quad \text{ $k$ is even,}
\end{aligned}
\right.
\end{equation*}
positive zeros , where $[x]$ means the maximum integer not more than $x$.
\end{corollary}
\begin{proof}
Note that
$$
\begin{aligned}
\mathcal{S}_{n+1,k}^{(1,1,m)}&=\bigcup_{j=1}^k\{i-j(m+1)|i\in\mathcal{T}_{j\cdot (n+1)}^{(1,1)}\setminus\mathcal{T}_{j-1}^{(1,1)}\} \\
&=\bigcup_{j=1}^k ([-j\cdot m, j\cdot (n-m)]\cap\mathbb{Z}) \\
&=[-k\cdot m, k\cdot (n-m)]\cap\mathbb{Z}.
\end{aligned}
$$
By Corollary \ref{coro1}, $m$ is odd. Then we have
$$
\begin{aligned}
\sharp\tilde{\mathcal{S}}_{n+1,k}^{(1,1,m)}&=\sharp(\mathcal{S}_{n+1,k}^{(1,1,m)} \cap 2\mathbb{Z})\\
&=\sharp([-k\cdot m, k\cdot (n-m)]\cap 2\mathbb{Z})=\left\{
\begin{aligned}
&[\frac{kn+1}{2}],\ \text{$k$ is odd}, \\
&\frac{kn}{2}+1,\ \text{$k$ is even}.
\end{aligned}
\right.
\end{aligned}
$$
Hence, $M_k(h)$ has at most
$$
\sharp\tilde{\mathcal{S}}_{n+1,k}^{(1,1,m)}-1=\left\{
\begin{aligned}
&[\frac{kn-1}{2}],\ \text{$k$ is odd}, \\
&\frac{kn}{2},\ \text{$k$ is even}.
\end{aligned}
\right.
$$
positive zeros.
\end{proof}

\subsection{The first order Melnikov function of perturbated quasi-homogeneous polynomial vector fields}
In this subsection, we consider the number of zeros of the first order Melnikov functions. In fact, this case is considered in the paper \cite{Li}, and next result is a generalization of Theorem A in \cite{Li}.
\begin{corollary}\label{firstMelnikov}
Assume that $(s_1,s_2)$-quasi-homogeneous vector fields  of weight degree $m$ has a center at the origin $O$, and the degree of  perturbated polynomials $P_i(x,y)$ and $Q_i(x,y)$ in system \eqref{eq2} are  not greater than $n$ and $n\ge m$. If the first order Melnikov $M_1$ does not vanish, then $M_1(h)$ has at most
$$
K(n,s_1)=\left\{
\begin{aligned}
&\frac{1}{4}(2s_1(n+1)-s_1^2+(3+(-1)^{n+1})s_1-7),\\
&\text{if $n+1\geq s_1$, $s_1$ and $s_2$ are both odd}; \\
&\frac{1}{4}((n+1)^2+(3+(-1)^{n+1})(n+1)+\frac{1}{2}(3\cdot(-1)^{n+1}-11)),\\
&\text{if $n+1< s_1$, $s_1$ and $s_2$ are both odd};\\
&\frac{1}{4}(2s_1(n+1)-s_1^2+2s_1-4+(-1)^{n+1}\cdot\frac{(-1)^{s_1}-1}{2}),\\
&\text{if $n+1\geq s_1$, $s_1+s_2$ is odd};\\
&\frac{1}{4}((n+1)^2+2(n+1)-\frac{1}{2}(7+(-1)^{n+1})),\quad \text{if $n+1< s_1$, $s_1+s_2$ is odd}
\end{aligned}
\right.
$$
zeros on $(0,+\infty)$.
\end{corollary}

We first provide a technical result as follows. Some of the situations (e.g. $r\geq s_1$) in following result have been achieved in Proposition 1 in \cite{Li}. For the convenience of readers, we will provide complete proof of the following lemma.
\begin{lemma}
\begin{itemize}
\item[(i)] if $s_1$ and $s_2$ are both odd, then
$$
\sharp (\mathcal{T}_r^{(s_1,s_2)}\cap(2\mathbb{Z}))=\left\{
\begin{aligned}
&\frac{1}{4}(2rs_1-s_1^2+(3+(-1)^r)s_1+1),\quad r\geq s_1, \\
&\frac{1}{4}(r^2+(3+(-1)^r)r+\frac{1}{2}(5+3\cdot(-1)^r)),\quad r< s_1;
\end{aligned}
\right.
$$
\item[(ii)] if $s_1$ is odd and $s_2$ is even, then
$$
\sharp (\mathcal{T}_r^{(s_1,s_2)}-(2\mathbb{Z}))=\left\{
\begin{aligned}
&\frac{1}{4}(2rs_1-s_1^2+2s_1-(-1)^r),\quad r\geq s_1, \\
&\frac{1}{4}(r^2+2r+\frac{1}{2}(1-(-1)^r)),\quad r< s_1;
\end{aligned}
\right.
$$
\item[(iii)] if $s_1$ is even and $s_2$ is odd, then
$$
\sharp (\mathcal{T}_r^{(s_1,s_2)}-(2\mathbb{Z}))=\left\{
\begin{aligned}
&\frac{1}{4}(2rs_1-s_1^2+2s_1),\quad r\geq s_1, \\
&\frac{1}{4}(r^2+2r+\frac{1}{2}(1-(-1)^r)),\quad r< s_1.
\end{aligned}
\right.
$$
\end{itemize}
\end{lemma}
\begin{proof}
We only prove statement (i) since the approaches of proofs in statements (ii) and (iii) are similar. For integer $u$, we define a family of straight lines
$$
L_u=\{(i,j)\in\mathbb{R}^2:s_1i+s_2j=u\},
$$
and let $\Delta_r$ be the closed triangle region surrounded by lines $i=0,j=0$ and $i+j=r$. Then $\sharp (\mathcal{T}_r^{(s_1,s_2,m)}\cap(2\mathbb{Z}))$ is the number of line $L_u$ which satisfies $L_u\cap\Delta_r$ has at least one integer point (that is the point whose two coordinates are both integer) and $u$ is even.

An important observation needs to be pointed out. If $(i,j)=(i_0,j_0)$ is an integer point on line $L_u$, then all the other integer points on $L_u$ can be represented as
$$
(i,j)=(i_0+ks_2,j_0-ks_1),\quad k\in\mathbb{Z},
$$
since $s_1,s_2$ are coprime. Thus, the distance between two adjacent integer points on line $L_u$ is equal to $\sqrt{s_1^2+s_2^2}$. This fact implies line $L_u$ has at least one integer point in the interior of $\Delta_r$ if the length of the segment of $L_u\cap\Delta_r$ is greater than $\sqrt{s_1^2+s_2^2}$.

In the following, we divide the proof of statement (i) into two cases.

\textbf{Case (i): $r\geq s_1$.} It is straightforward to verify that there are two unique lines $L_{s_1s_2}$ and $L_{s_1(r-(s_1-s_2))}$ such the length of segments $L_u\cap\Delta_r$ is equal to $\sqrt{s_1^2+s_2^2}$. When $r\geq s_1$, the lines $L_u$ which intersect with the region $\Delta_r$ have three types:
\begin{itemize}
\item[I.] the lines $L_u$ with $u=s_1s_2,s_1s_2+1,\dots,s_1(r-(s_1-s_2))$;
\item[II.] the lines $L_u$ with $u=0,1,\dots,s_1s_2-1$;
\item[III.] the lines $L_u$ with $u=s_1(r-(s_1-s_2))+1,s_1(r-(s_1-s_2))+2,\dots,rs_1$.
\end{itemize}
We first calculate the number of lines of type I. It is clear that the length of the segments of lines of type I in $\Delta_r$ is greater than or equal to $\sqrt{s_1^2+s_2^2}$. Thus, the lines of type I always have integer points in $\Delta_r$. It implies that all the even integers in the set $\{s_1s_2,s_1s_2+1,\dots,s_1(r-(s_1-s_2))\}$ belong to the set $\mathcal{T}_r^{(s_1,s_2,m)}\cap(2\mathbb{Z})$, and the number of these integers is equal to
$$
N_I=\frac{s_1(r-(s_1-s_2))-s_1s_2+\frac{1+(-1)^r}{2}}{2}=\frac{2s_1(r-s_1)+1+(-1)^r}{4}.
$$

Next we calculate the number of lines of type II. Let
$$
\delta=\{(i,j)\in\mathbb{R}^2|\ 0\leq i<s_2,\ 0\leq j<s_1,\ s_1i+s_2j<s_1s_2\},
$$
Note that each lines $L_u$ of type II passing through different integer points in $\delta$ determine different values of $u\in\{0,1,\dots,s_1s_2-1\}$. Otherwise, there exist two integer points $(i,j)=(i_1,j_1),(i_2,j_2)$ lie on a same line $L_{u_0}$, then the distance of such two integer points are greater than or equal to $\sqrt{s_1^2+s_2^2}$. It implies $u_0\geq s_1s_2$,  which is a contradiction with the fact $L_{u_0}\cap\delta\neq\emptyset$. Next we only need to calculate the number of integer points $(i,j)$ which satisfy $s_1i+s_2j$ is even in $\delta$.

We first claim there always exists an integer point on $L_u$ in $\delta$ for any $u=s_1s_2-s_2,s_1s_2-s_2+1,\dots,s_1s_2-1$.
If the claim is not true, then there exists a $L_{u_0}$ such that $L_{u_0}\cap\delta=\emptyset$, where $s_1s_2-s_2\leq u_0\leq s_1s_2-1$. Consider the parallelogram
$$
\delta^*=\{(i,j)\in\mathbb{R}^2|\ 0\leq j< s_1,\ 0\leq s_1i+s_2j< s_1s_2\}.
$$
Since $L_{u_0}\cap\delta^*$ has at least one integer point, $L_{u_0}$ shall have an integer point in the region
$$
(\bigcup_{s_1s_2-s_2\leq u\leq s_1s_2-1}L_u)\bigcap(\delta^*\setminus\delta)=\{(i,j)\in\mathbb{R}^2|\ i<0,\ j< s_1,\  s_1i+s_2j\geq s_1s_2-s_2\}.
$$
It implies that  $s_1-1< j< s_1$,  which leads to a contradiction. $L_u$ for $u=s_1s_2-s_2,s_1s_2-s_2+1,\dots,s_1s_2-1$ contains exactly
$$
\frac{1}{2}(s_1s_2-1-(s_1-1)s_2)+1=\frac{1}{2}(s_2-1)+1
$$
lines with even $u$.

Now we consider the lines $L_u$ of type II with $u=0,1,\dots,s_1s_2-s_2-1$. It is clear that there are exactly $\frac{1}{2}(s_1s_2-s_2)$ lines $L_u$ of type II with even $u$, and $\frac{1}{2}(s_1-1)$ of them pass through the point $(i,j)=(0,0),(0,2),(0,4),\dots,(0,s_1-3)$ on the $j$-axis. The left $\frac{1}{2}(s_1s_2-s_2)-\frac{1}{2}(s_1-1)$ lines of type II intersect $\delta^*$ except $j$-axis with exactly one integer point which belongs to either
$$
\delta_1=\delta\bigcap \{(i,j)\in\mathbb{R}^2|\ i>0,\ s_1i+s_2j\leq s_1s_2-s_2\}
$$
or
$$
\delta_2=(\delta^*\setminus\delta)\bigcap\{(i,j)\in\mathbb{R}^2|\ j\leq s_1-1\},
$$
and these two regions are symmetric about the point $(i,j)=(0,\frac{1}{2}(s_1-1))$.

It is straightforward to verify that $(-i_0,s_1-1-j_0)\in\delta_2$ is an integer point on some $L_u$ with even $u$ if $(i_0,j_0)\in \delta_1$ is an integer point such that $s_1i_0+s_2j_0$ is even. Hence, half of these intersection points are in the interior of $\delta_1$ and the other half are in the interior of $\delta_2$ which means there are only
$$
\frac{1}{2}(\frac{1}{2}(s_1s_2-s_2)-\frac{1}{2}(s_1-1))=\frac{1}{4}(s_1-1)(s_2-1)
$$
left lines having integer points in $\delta$. To summarize, there exist exactly
$$
N_{II}=\frac{1}{2}(s_2-1)+1+\frac{1}{2}(s_1-1)+\frac{1}{4}(s_1-1)(s_2-1)=\frac{1}{4}(s_1+1)(s_2+1)
$$
lines of type II with even $u$ having an integer point in $\delta$.

At last, we calculate the number of lines of type III. Let
$$
\Delta=\{(i,j)\in\mathbb{R}^2|\ j\geq 0,\ i+j\leq r,\ s_1i+s_2j>s_1(r-(s_1-s_2))\},
$$
It is easily verified that each lines $L_u$ of type III passing through different integer points in $\Delta$ determine different values of $u\in\{s_1(r-(s_1-s_2))+1,s_1(r-(s_1-s_2))+2,\dots,rs_1\}$.
Let
$$
\Delta^*=\{(i,j)\in\mathbb{R}^2|\ 0\leq j<s_1,\ s_1(r-(s_1-s_2))<s_1i+s_2j<rs_1 \}.
$$
It is easily verified that there are exactly $\frac{1}{2}s_1(s_1-s_2)$ lines $L_u$ of type III with even $u$, and these lines has exactly one integer points in $\Delta^*$.

If $r$ is even, then $\frac{1}{2}(s_1-s_2)$ of above lines pass through integer points $(r-(s_1-s_2)+2,0),(r-(s_1-s_2)+4,0),\dots,(r,0)$ on the $i$-axis, and $s_1-1$ of above lines pass through integer points $(r-j,j)$ on line $i+j=r$ with $j=1,2,\dots,s_1-1$. The number of these two kind of lines is $\frac{1}{2}(s_1-s_2)+s_1-1$.

If $r$ is odd, then  $\frac{1}{2}(s_1-s_2)$ of above lines pass through integer points $(r-(s_1-s_2)+1,0),(r-(s_1-s_2)+3,0),\dots,(r-1,0)$ on the $i$-axis, and none of above lines pass through integer points $(r-j,j)$ on line $i+j=r$ with $j=1,2,\dots,s_1-1$. The number of these lines is $\frac{1}{2}(s_1-s_2)$.

Let us consider the line with even $u$ which only pass through the integer points in the interior of $\Delta$. Note that the number of lines $L_u$ passing through the integer points in $\Delta^*$ except the line $i+j=r$ and $i$-axis is equal to $\frac{1}{2}s_1(s_1-s_2)-(\frac{1}{2}(s_1-s_2)+\frac{1+(-1)^r}{2}(s_1-1))$. It is no hard to see that the integer point $(2r-s_1-i_0,s_1-j_0)$ determine a line $L_u$ with some even $u$ if $(i_0,j_0)$ is another integer point such that $s_1i_0+s_2j_0$ is even. Hence, half of these $\frac{1}{2}s_1(s_1-s_2)-(\frac{1}{2}(s_1-s_2)+\frac{1+(-1)^r}{2}(s_1-1))$ line pass through the integer points in the interior of $\Delta$ and the other half pass through the integer points in the interior of $\Delta\setminus\Delta^*$. To summarize, there exist exactly
$$
\begin{aligned}
N_{III}&=\frac{1}{2}(\frac{1}{2}s_1(s_1-s_2)-(\frac{1}{2}(s_1-s_2)+\frac{1+(-1)^r}{2}(s_1-1)))+\frac{1}{2}(s_1-s_2)\\
&+\frac{1+(-1)^r}{2}(s_1-1)\\
&=\frac{1}{4}((s_1+1)(s_1-s_2)+(1+(-1)^r)(s_1-1))
\end{aligned}
$$
lines of type III with even $u$ having an integer point in $\Delta$.

To sum up, if $r\geq s_1$, then
$$
\begin{aligned}
\sharp (\mathcal{T}_r^{(s_1,s_2)}\cap(2\mathbb{Z}))&=N_I+N_{II}+N_{III}\\
&=\frac{1}{4}(2rs_1-s_1^2+(3+(-1)^r)s_1+1).
\end{aligned}
$$

\textbf{Case (ii): $r<s_1$.} It is clear that different integer points in $\Delta_r$ determine different line $L_u$. Since $s_1$ and $s_2$ are both odd, $s_1i+s_2j$ is even if and only if $i+j$ is even. Hence, we only need to count the number of integer point in $\Delta_r$ which satisfy $i+j$ is even.

Note that the integer points on $L_{u}\cap\Delta_r$ are $(l-j,j)$ with $j=0,1,2,\dots,l$ (the total number is $l+1$) and $l=0,1,\dots,r$. If $r$ is even, then the number of integer point in $\Delta_r$ with even $i+j$ is
$$
1+3+5+\dots+(r+1)=\frac{1}{4}(r+2)^2.
$$
If $r$ is odd, then the number of integer point in $\Delta_r$ with even $i+j$ is
$$
1+3+5+\dots+(r-1+1)=\frac{1}{4}(r+1)^2.
$$
To summarize, if $r< s_1$, then
$$
\sharp (\mathcal{T}_r^{(s_1,s_2,m)}\cap(2\mathbb{Z}))=\frac{1}{4}(r^2+(3+(-1)^r)r+\frac{1}{2}(5+3\cdot(-1)^r)).
$$
\end{proof}

Next we finish the proof of Corollary \ref{firstMelnikov}.

\begin{proof}[Proof of Corollary \ref{firstMelnikov}]
We only prove the conclusion in case $s_1$ and $s_2$ are both odd, and the others can be proved using the similar arguments. For given odd $s_1,s_2,m$ (recall $s_1+s_2+m-1$ is even), by Theorem \ref{mainthm}, the number of positive zero of $M_1(h)$ is less than or equal to
$$
\begin{aligned}
\sharp\tilde{S}^{(s_1,s_2,m)}_{n+1,1}-1&= \sharp(\{p-(s_1+s_2+m-1)|p\in\mathcal{T}_{n+1}^{(s_1,s_2)}\setminus\{0\}\}\cap(2\mathbb{Z}))-1\\
&=\sharp(\mathcal{T}_{n+1}^{(s_1,s_2)}\cap(2\mathbb{Z}))-2\\
&=K(n,s_1).
\end{aligned}
$$
\end{proof}

\subsection{Higher order Melnikov functions of perturbated quasi-homogeneous polynomial systems}
Finally, we extend the results  in \cite{FHX} from the quasi-homogeneous Hamiltonian case  to the general quasi-homogeneous case. We obtain an upper bounded on the number of zeros of the first non-vanishing $k$th order Melnikov functions $M_k(h)$ of system \eqref{eq2} as follows.
\begin{theorem}\label{mainthm1}
Assume that system \eqref{eq1} is $(s_1,s_2)$-quasi-homogeneous with weight degree $m$, and the degree of  perturbated system \eqref{eq2}
is $n$ and $n\geq m$. If system \eqref{eq1} has a center at the origin $O$, then the first non-vanishing $k$th order Melnikov function $M_k(h)$ of system \eqref{eq2} has at most
\begin{equation*}
K(k,n,s_1,s_2)=\left\{
\begin{aligned}
&\frac{k((n+1)s_1-s_2)}{2}-1,\quad
\begin{aligned}
&\text{$n$ is even, and all $k,s_1,s_2$ are odd,}\\
&\text{$k$ is even, and $s_1+s_2$ is odd,}\\
&\text{$s_2$ is even, and all $k,s_1,n$ is odd,}
\end{aligned}\\
&\frac{k((n+1)s_1-s_2)}{2},\quad \quad \text{ $k$ is even, and both $s_1$ and $s_2$ are odd,}\\
&\frac{k((n+1)s_1-s_2)-1}{2},\quad \text{others},
\end{aligned}
\right.
\end{equation*}
positive zeros.
\end{theorem}
\begin{proof}
We only prove this lemma in case $s_1$ and $s_2$ are both odd since the approaches of proofs in the remaining cases are similar. When $n\geq m$, the maximal and minimal number in $\mathcal{S}_{n+1,k}^{(s_1,s_2,m)}$ are $k((n+1)s_1-(s_1+s_2+m-1))$ and $k(s_2-(s_1+s_2+m-1))$ respectively. Hence,
$$
\mathcal{S}_{n+1,k}^{(s_1,s_2,m)}\subseteq [k(s_2-(s_1+s_2+m-1)),k((n+1)s_1-(s_1+s_2+m-1))]\cap\mathbb{Z}.
$$
Recall $s_1+s_2+m-1$ is even by Corollary \ref{coro1}, then we have $s_2-(s_1+s_2+m-1)$ is odd. Hence
$$
\begin{aligned}
\sharp\tilde{\mathcal{S}}_{n+1,k}^{(s_1,s_2,m)}&\leq\sharp ([k(s_2-(s_1+s_2+m-1)),k((n+1)s_1-(s_1+s_2+m-1))]\cap2\mathbb{Z})\\
&=\left\{
\begin{aligned}
&[\frac{k((n+1)s_1-s_2)+1}{2}],\quad \text{ $k$ is odd,}\\
&\frac{k((n+1)s_1-s_2)}{2}+1,\quad \text{ $k$ is even.}
\end{aligned}
\right.
\end{aligned}
$$
If $k$ is even, by Theorem \ref{mainthm}, $M_k(h)$ has at most $\sharp\tilde{\mathcal{S}}_{n+1,k}^{(s_1,s_2,m)}-1=\frac{k((n+1)s_1-s_2)}{2}$ positive zeros; If $k$ is odd and $n$ is even, then $(n+1)s_1-s_2$ is even. Hence, $M_k(h)$ has at most $\sharp\tilde{\mathcal{S}}_{n+1,k}^{(s_1,s_2,m)}-1=\frac{k((n+1)s_1-s_2)}{2}-1$ positive zeros. Similarly, if $k$ and $n$ are both odd, then $(n+1)s_1-s_2$ is odd and $M_k(h)$ has at most $\sharp\tilde{\mathcal{S}}_{n+1,k}^{(s_1,s_2,m)}-1=\frac{k((n+1)s_1-s_2)+1}{2}-1$ positive zeros.
\end{proof}

{\bf Remark}: Theorem \ref{mainthm1} and Theorem 5 in \cite{FGX} imply that the cyclicity of the period
annulus of any quasi-homogeneous polynomial vector fields is finite. Since system \eqref{eq2} is any quasi-homogeneous polynomial vector field in $\mathbb{R}^2$,  the upper bounded in Theorem  \ref{mainthm1} is usually not sharp.
In the last section we will provide an example to show that. Therefore, for a given quasi-homogeneous polynomial vector field with a center, finding the exact upper bound on the number of isolated zeros of any $k$-th Melnikov function remains an interesting and worthy topic for further research.

\section{Application}
In the last, we provide an example to explain how our algorithm works in the bifurcation problem of perturbated quasi-homogeneous polynomial foliations. Consider the perturbated quasi-homogeneous Hamiltonian foliation
\begin{equation}\label{pullback_eightloop}
\omega_\epsilon=dH+\epsilon \omega_1=0,\quad 0<\epsilon\ll 1
\end{equation}
where $H=\frac{1}{4}x^4+\frac{1}{2}y^2$ is a $(1,2)-$quasi-homogeneous polynomial of weight degree four, and
$$
\omega_1=(c_0+c_1x+c_2x^2+c_3y+c_4x^3+c_5xy)dx+(b_0+b_1x+b_2x^2+b_3y)dy,
$$
which includes all 1-forms whose degree is  less than the degree of  quasi-homogeneous  1-form $dH$.

Under the generalized polar coordinates transformation $\Phi(r,\theta)=(rCs\theta,r^2Sn\theta)$, the foliation \eqref{pullback_eightloop} has the form
\begin{equation}\label{pullback_eightloop_polar}
dr+\epsilon(F(r,\theta)dr+G(r,\theta)d\theta)=0,
\end{equation}
where
$$
\begin{aligned}
F(r,\theta)=&r^{-3}\cdot c_0 Cs\theta + r^{-2}\cdot (c_1 Cs^2\theta + 2b_0 Sn\theta)+r^{-1}\cdot ((2b_1+c_3)Cs\theta Sn\theta \\
&+ c_2 Cs^3\theta)+c_4 Cs^4\theta + (c_5 + 2 b_2) Cs^2\theta Sn\theta + 2 b_3 Sn\theta ,\\
G(r,\theta)=&r^{-2}\cdot (-c_0 Sn\theta) + r^{-1} \cdot (b_0 Cs^3\theta - c_1 Cs\theta Sn\theta) + b_1 Cs^4\theta - c_2 Cs^2\theta\\
&\cdot Sn\theta - c_3 Sn^2\theta + r \cdot (b_2 Cs^5\theta + (b_3-c_4) Cs^3\theta Sn\theta - c_5 Cs\theta Sn^2\theta).
\end{aligned}
$$
Note that the period annulus is the union of level set $r=h$, $h>0$ for foliation \eqref{pullback_eightloop_polar}. Let us consider the bifurcation of limit cycles from the  period annulus of foliation \eqref{pullback_eightloop_polar}.


\begin{theorem}
The number of limit cycles bifurcated from the open period annulus $\mathbb{R}^2\setminus\{(0,0)\}$ of Hamiltonian foliation $d(\frac{1}{4}x^4+\frac{1}{2}y^2)=0$ with respect to one-parameter linear deformation \eqref{pullback_eightloop} is equal to one.
More precise, the following statements hold.
\begin{itemize}
\item $M_1(h)$ has no positive zero if $M_1(h)\not\equiv 0$;
\item $M_2(h)$ has a unique  positive zero if $M_1(h)\equiv 0$ and $M_2(h)\not\equiv 0$;
\item $M_3(h)$ has a unique  positive zero if  $M_1(h)=M_2(h)\equiv 0$ and $M_3(h)\not\equiv 0$;
\item $M_k(h)\equiv 0$ for all $k\ge 4$ if  $M_1(h)=M_2(h)=M_3(h)\equiv 0$.
\end{itemize}
Futhermore, the foliation \eqref{pullback_eightloop} always possesses an open period annulus $\mathbb{R}^2\setminus U_{(0,0)}$ if $M_1(h)=M_2(h)=M_3(h)\equiv 0$, where $U_{(0,0)}$ is a small neighborhood of the origin. The algebraic set $M_1(h)=M_2(h)=M_3(h)\equiv 0$ has three irreducible components as follows
\begin{itemize}
  \item[($C_1$).] $c_3-b_1=c_5-2b_2=0$;
  \item[($C_2$).] $c_0=c_2=c_3=b_1=0$;
  \item[($C_3$).] $b_1-c_3=c_0=c_2=b_0=c_5+b_2=0$.
\end{itemize}
The foliation \eqref{pullback_eightloop} is Hamiltonian on the component ($C_1$), and  foliation \eqref{pullback_eightloop} can be pulled back to a symmetric foliation on the components ($C_2$) and ($C_3$).
\end{theorem}

\begin{proof}
This result is an application of the algorithm in Theorem \ref{thm-Fran3}. We  show the necessary process of the calculation since the integration of generalized trigonometric functions is pretty tedious.

By directly computation,  we have
$$
M_1(h)=-\int_0^\tau G(h,\theta)d\theta=\frac{1}{3}(c_3-b_1)\tau.
$$
Thus, $M_1(h)$ has no zeros on $(0,\ +\infty)$.

When $M_1(h)\equiv 0$, i.e. $c_3-b_1=0$, by Theorem \ref{thm-Fran3} we can calculate
$$
\begin{aligned}
a_1(r,\theta)=&-r^{-3}\cdot c_0 Cs\theta - r^{-2} \cdot (\frac{3}{2}c_1 Cs^2\theta+3b_0 Sn\theta) - r^{-1} (c_2 Cs^3\theta + 3 c_3 Cs\theta Sn\theta) \\
&- (\frac{3}{4}c_4 + \frac{1}{4} b_3) Cs^4 \theta - (\frac{3}{4} c_5 +\frac{3}{2} b_2) Cs^2\theta Sn\theta -2 b_3 Sn^2\theta - (\frac{1}{4}a_5 -\frac{1}{2} b_2) \int_0 ^\theta Cs t dt\\
\end{aligned}
$$
by the equation $d(F(r,\theta)dr+G(r,\theta)d\theta)=-da_1\wedge dr$. We obtain that
$$
M_2(h)=-\int_0^\tau a_1(h,\theta)G(h,\theta)d\theta=-\frac{1}{9h^2}(c_5-2b_2)(9c_0\int_0^\tau Cs^2 \theta d\theta+c_2\tau h^2),
$$
which has at most one positive zero, and this bound can be reached.

The zero locus of $M_1(h)=M_2(h)=0$ has two irreducible components:
\begin{itemize}
  \item[($C_1$)] $c_3-b_1=0,c_5-2b_2=0$;
  \item[($C_2$)] $c_3-b_1=0, c_0=c_2=0$.
\end{itemize}
The foliation lies in the first component is Hamiltonian since if $c_3-b_1=0,c_5-2b_2=0$, the 1-form
$$
\omega_\epsilon=d(\frac{1}{4}x^4+\frac{1}{2}y^2+\epsilon(c_0x+\frac{1}{2}c_1+\frac{1}{3}c_2 x^3 + b_1 xy + b_2 x^2y + \frac{1}{2}b_3 y^2))=0.
$$
It  is exact which implies $M_i(h)\equiv 0$ for any $i\geq 1$; When the foliation \eqref{pullback_eightloop_polar} lies in the second component, then by using the similar arguments, we have
$$
\begin{aligned}
a_2(r,\theta) = &\frac{1}{480r^4}\Bigl(720c_1^2Cs\theta^4 + 2880c_1b_0Cs\theta^2Sn\theta + 2880b_0^2Sn\theta^2 \\
&\quad + 540Cs\theta^2(c_1^2Cs\theta^2-2b_0^2Cs\theta^2+4c_1b_0Sn\theta)\Bigr) \\
+ &\frac{1}{480r^3}\Bigl(3600c_1c_3Cs\theta^3Sn\theta + 7200c_3b_0Cs\theta Sn\theta^2 \\
&\quad + 96(15c_1c_3Cs\theta^3Sn\theta + 6c_3b_0Cs\theta(2-2Cs\theta^4+Sn\theta^2))\Bigr) \\
+ &\frac{1}{480r}\Bigl(1440c_3c_4Cs\theta^5Sn\theta + 1440c_3Cs\theta Sn\theta^2(c_5Cs\theta^2+2b_2Cs\theta^2+2b_3Sn\theta) \\
&\quad + 360c_3Cs\theta Sn\theta(3c_4Cs\theta^4+3c_5Cs\theta^2Sn\theta+6b_2Cs\theta^2Sn\theta \\
&\quad + b_3(Cs\theta^4+8Sn\theta^2)+c_5\textstyle\int_0^\theta Cst dt-2b_2\textstyle\int_0^\theta Cst dt)\Bigr) \\
+ &\frac{1}{480r^2}\Bigl(720c_1c_4Cs\theta^6 + 1440c_4b_0Cs\theta^4Sn\theta + 4320c_3^2Cs\theta^2Sn\theta^2 \\
&\quad + 720c_1Cs\theta^2Sn\theta(c_5Cs\theta^2+2b_2Cs\theta^2+2b_3Sn\theta) \\
&\quad + 1440b_0Sn\theta^2(c_5Cs\theta^2+2b_2Cs\theta^2+2b_3Sn\theta) \\
&\quad + 120c_1Cs\theta^2(3c_4Cs\theta^4+3c_5Cs\theta^2Sn\theta+6b_2Cs\theta^2Sn\theta \\
&\quad + b_3(Cs\theta^4+8Sn\theta^2)+c_5\textstyle\int_0^\theta Cst dt-2b_2\textstyle\int_0^\theta Cst dt) \\
&\quad + 240b_0Sn\theta(3c_4Cs\theta^4+3c_5Cs\theta^2Sn\theta+6b_2Cs\theta^2Sn\theta \\
&\quad + b_3(Cs\theta^4+8Sn\theta^2)+c_5\textstyle\int_0^\theta Cst dt-2b_2\textstyle\int_0^\theta Cst dt)\Bigr) \\
+ &\frac{1}{480}\Bigl(120c_4Cs\theta^4(3c_4Cs\theta^4+3c_5Cs\theta^2Sn\theta+6b_2Cs\theta^2Sn\theta \\
&\quad + b_3(Cs\theta^4+8Sn\theta^2)+c_5\textstyle\int_0^\theta Cst dt-2b_2\textstyle\int_0^\theta Cst dt) \\
&\quad + 120Sn\theta(c_5Cs\theta^2+2b_2Cs\theta^2+2b_3Sn\theta) \\
&\quad \times (3c_4Cs\theta^4+3c_5Cs\theta^2Sn\theta+6b_2Cs\theta^2Sn\theta \\
&\quad + b_3(Cs\theta^4+8Sn\theta^2)+c_5\textstyle\int_0^\theta Cst dt-2b_2\textstyle\int_0^\theta Cst dt)\Bigr)
\end{aligned}
$$
and
$$
M_3(h)=-\int_0^\tau a_2(h,\theta)G(h,\theta)d\theta=\frac{1}{9h^2}c_3(c_5-2b_2)(9b_0\int_0^\tau Cs^2\theta d\theta +\tau (c_5+b_2)h^2),
$$
which has at most one positive zero, and this bound can be reached.

The zero locus of $M_1(h)=M_2(h)=M_3(h)\equiv 0$ has three irreducible components:
\begin{itemize}
  \item[($C_1$)] $c_3-b_1=c_5-2b_2=0$;
  \item[($C_2$)] $c_0=c_2=c_3=b_1=0$;
  \item[($C_3$)] $b_1-c_3=c_0=c_2=b_0=c_5+b_2=0$.
\end{itemize}
We have shown that the first one is Hamiltonian. The foliation lies on the second component is actually symmetric with respect to $x=0$. In other words, if $c_0=c_2=c_3=b_1=0$, the foliation determined by \eqref{pullback_eightloop}
$$
\omega_\epsilon=d(\frac{1}{4}x^4+\frac{1}{2}y^2)+\epsilon(\frac{1}{2}(c_1+c_4x^2+c_5y)d(x^2)+(b_0+b_2x^2+b_3y)dy)=0
$$
is invariant under the map $(x,y)\mapsto (-x,y)$, which implies the foliation determined by \eqref{pullback_eightloop} has a family of closed orbits.

The foliation lies on the third component is determined by the equation
$$
\omega_\epsilon=((1+\epsilon c_4) x^3 + \epsilon c_1 x + \epsilon y ( b_1-b_2 x)) dx+(\epsilon b_1 x + \epsilon b_2 x^2 + (1 + \epsilon b_3) y) dy = 0.
$$
Let $z = \epsilon b_1 x + \epsilon b_2 x^2 + (1 + \epsilon b_3) y$, then the equation $\omega_\epsilon=0$ can be transformed into
$$
(f(x)-3\epsilon b_2xz)dx+zdz=0,
$$
where $f(x)=(((1+\epsilon c_4)(1+\epsilon b_3)+\epsilon^2 b_2^2)x^3+\epsilon(c_1(1+\epsilon b_3)-\epsilon b_1^2)x$. This equation is also invariant under the map $(x,z)\mapsto (-x,z)$, which implies the foliation determined by above equation also has a family of closed orbits. To sum up, we know the foliation determined by \eqref{pullback_eightloop} always has a period annulus if $M_1(h)=M_2(h)=M_3(h)\equiv 0$, which implies the Melnikov function of the $k$th order  vanishes for every $k\geq 4$.

\end{proof}

\section*{Acknowledgements}

The work is partially supported by National Key R $\&$ D Program of China (No. 2022YFA1005900).
H. He is also partially supported by the National Natural
Science Foundations of China  (No. 12401212) and the Fundamental Research Funds for the Central Universities (Project No. 2025CDJ-IAISYB-013). D. Xiao is also partially supported by the National Natural
Science Foundations of China (No. 12271353).


\begin{thebibliography}{99}

\bibitem{arn80} Arnol'd I. V..
 {\it Ordinary differential equations. {Transl}. from the 3rd
  {Russian} ed. by {Roger} {Cooke}}.
 Springer-Textb. Berlin etc.: Springer-Verlag, 1992.


\bibitem{BZ}{ Berezin I.S., Zhidkov N.P.}. {\it Computing Methods, vol. II.} Pergamon Press, Oxford (1964).

\bibitem{BGL}{ Buica A., Gin\'e J. , Llibre J.}.
{\it Bifurcation of limit cycles from a polynomial degenerate center.}  Adv. Nonlinear Stud. 10 (2010): 597 - 609.

\bibitem{fors37}
 Forster H..
{\it {\"U}ber das {Verhalten} der {Integralkurven} einer gew{\"o}hnlichen
  {Differentialgleichung} erster {Ordnung} in der {Umgebung} eines
  singul{\"a}ren {Punktes}.}
 Math. Z., 43 (1937): 271 - 320.

\bibitem{F}{Fran\c{c}oise J.-P.}. {\it Successive derivatives of a first return map, application to the study of quadratic vector fields}. Ergodic Theory Dynam. Syst., 16(1996).

\bibitem{FHX} {Francoise J.-P., He H. , Xiao D.}.  {\it The number of limit cycles bifurcating from the period annulus of quasi-homogeneous Hamiltonian systems at any order.} Journal of Differential Equations, 276 (2021): 1 - 24.

 \bibitem{FP}{  Fran\c{c}oise J.-P.,  Pelletier M.}.
{\it Iterated integrals, Gelfand-Leray residue, and first return mapping.} J. Dyn. Control Syst. 12 (2006): 357 - 369.

\bibitem{FGX}{ Fran\c{c}oise J.-P., Gavrilov  L.,  Xiao D.}.
{\it Hilbert's 16th problem on a period annulus and Nash space of arcs}. Math. Proc. Cam. Phil. Soc. 169 (2020): 377 - 409.






\bibitem{gavr08}
Gavrilov L..
{\it Cyclicity of period annuli and principalization of {B}autin ideals.}
 Ergodic Theory Dynam. Systems, 28 (2008):1497 - 1507.

\bibitem{gavr13}
 Gavrilov L..
 {\it On the number of limit cycles which appear by perturbation of
  two-saddle cycles of planar vector fields}.
Funct. Anal. Appl., 47 (2013):174 - 186.

\bibitem{gano08}
Gavrilov L., Novikov D. .
{\it On the finite cyclicity of open period annuli}.
Duke Math. J. , 152 (2010):1 - 26.


\bibitem{G}{Gavrilov L.}.
{\it Higher order Poincare-Pontryagin functions and iterated path integrals}.  Ann.fac.sci.toulouse Math XIV(2005):663 - 682.






\bibitem{Iliev}{Iliev I.}.
{\it The number of limit cycles due to polynomial perturbations of the harmonic oscillator.} Math. Proc. Cambridge Philos. Soc., 127 (1999): 317 - 322.

\bibitem{Iliev2}{Iliev I. D., Li C., Yu J.}.
{\it On the cubic perturbations of the symmetric 8-loop Hamiltonian.} Journal of Differential Equations 269(2020): 3387 - 3413.

\bibitem{LianLiu}{Lian H., Liu C,, Yang J.}.
{\it On the cyclicity of quasi-homogeneous polynomial systems},  Journal of Mathematical Analysis and Applications 516(2022): 126510.

\bibitem{Ilyashenko} {Ilyashenko S. Yu. ,} {\it The origin of limit cycles under perturbation of the equation
dw/dz = -Rz/Rw, where R(z, w) is a polynomial}. Math. USSR Sbornik, 78 (1969): 360šC373.

\bibitem{IY} Ilyashenko S. Yu., Yakovenko S., {\it Finite cyclicity of elementary polycycles in generic families,}
in: Concerning the Hilbert 16th Problem, Amer. Math. Soc. Transl., vol. 165, Amer. Math.
Soc., Providence, RI (1995): 21–95.

\bibitem{L}{Lyapunov A.M.}.
{\it Stability of motion.} In: Mathematics in Science and Engineering,
vol. 30. Academic Press, San Diego, 1966.

\bibitem{JMP}{Jebrane  A., Mardesic P.,  Pelletier M.}.
 {\it A generalization of Fran\c{c}oise's algorithm for calculating higher order Melnikov functions}.
 Bull. Sci. math., 126(2002): 705 - 732.

\bibitem{JZ}{ Jebrane A. and  Zoladek H.}.
{\it A note on higher order Melnikov functions.} Qualitative Theory of Dynamical Systems 6.2(2005):273 - 287.

\bibitem{Kal} Kaloshin V.. {\it The existential Hilbert 16-th problem and an estimate for cyclicity of elementary
polycycles,} Invent. Math., 151 (2003): 451–512.

\bibitem{LLNZ} {Li W., Llibre J., Nicolau M., Zhang X.}. {\it On the differentiability of first integrals of two dimensional flows.} Proc. Am. Math. Soc., 130 (2002): 2079 - 2088.

\bibitem{Li} { Li W.,  Llibre J., Yang J., Zhang Z.}.
{\it Limit cycles bifurcating from the period annulus of quasi-homogeneous centers}. J.
Dynamics and Differential Equations, 21 (2009): 133 - 152.

\bibitem{nest60}
 Nemytskii V. V. and Stepanov V.~V..
\newblock {\em Qualitative theory of differential equations}, volume~22 of {\em
  Princeton Math. Ser.}
\newblock Princeton University Press, Princeton, NJ, 1960.


\bibitem{Rou1} {Roussarie R.}. {\it Bifurcation of planar vector fields and Hilbert’s sixteenth problem}. vol. 164 of Progress in Mathematics. Birkhäuser Verlag, Basel, 1998.

\bibitem{Rou2} {Roussarie R.}. {\it Melnikov Functions and Bautin Ideal}. Qualitative Theory of Dynamical
 Systems, 2(2001): 67 - 78.

\bibitem{TangZhang} Tang Y., Zhang X.. {\it Global dynamics of planar quasi-homogeneous differential systems.} Nonlinear Analysis: Real World Applications, 49 (2019): 90 - 110.


\bibitem{ZhangYu}{ Zhang L., Yu J.}. {\it On the center criterion of planar quasi-homogeneous polynomial differential systems.} Bulletin des Sciences Math\'ematiques, 147 (2018): 7 - 25.

 \bibitem{ZhangLi}{Zhang Z., Li B. }.
{\it High order Melnikov functions and the problem of uniformity in global bifurcation.}
 Annali di Matematica pura ed applicata,  CLXI (1992): 181 - 212.



\end{thebibliography}
\end{document}